\documentclass[11pt,reqno]{amsart}
\usepackage{amsmath, amssymb, amsthm}
\usepackage[linktocpage=true]{hyperref}
\hypersetup{colorlinks,linkcolor=blue,urlcolor=blue,citecolor=blue}
\usepackage{mathtools} 
\usepackage[table]{xcolor}
\usepackage{stmaryrd} 
\usepackage{bbm} 
\usepackage{leftidx} 
\usepackage{mathabx} 
\usepackage{rotating} 
\usepackage{wasysym} 
\usepackage[mmddyy]{datetime}
\usepackage{qtree}
\usepackage[all]{xy}
\let\normallabel=\label 
\usepackage[bbgreekl]{mathbbol} 
\DeclareSymbolFontAlphabet{\mathbbm}{bbold}
\DeclareSymbolFontAlphabet{\mathbb}{AMSb}
\usepackage{arydshln}
\makeatletter
\renewcommand*\env@matrix[1][*\c@MaxMatrixCols c]{%
  \hskip -\arraycolsep
  \let\@ifnextchar\new@ifnextchar
  \array{#1}}
\makeatother
\usepackage{youngtab,young}

\newcommand{\iyoung}[1]{{\scriptsize\Yvcentermath1 \young(#1)}}
\makeatletter
\newcommand{\labeltarget}[1]{\Hy@raisedlink{\hypertarget{#1}{}}}
\makeatother
\usepackage{todonotes}
\usepackage{pgf, tikz, colortbl, braids}
\usetikzlibrary{cd, braids}
\numberwithin{equation}{subsection}
\usepackage[margin = 0.9in]{geometry}
\newcommand{\based}{ground }
\newcommand{\Based}{Ground }
\newcommand{\tcH}{\widetilde{\cH}}

\def\<{\langle}
\def\>{\rangle}
\DeclareMathSymbol{\hp}{\mathord}{AMSa}{"39}
\newcommand{\alphab}{\overline{\alpha}}

\newcommand{\cA}{\mathcal{A}}
\newcommand{\cB}{\mathcal{B}}

\newcommand{\cH}{\mathcal{H}}

\newcommand{\cM}{\mathcal{M}}

\newcommand{\CC}{\mathbb{C}}

\newcommand{\End}{\operatorname{End}}
\newcommand{\Ext}{\mrm{Ext}}

\newcommand{\ft}{\mathfrak{t}}

\newcommand{\GL}{\mrm{GL}}
\newcommand{\Hom}{\mrm{Hom}}

\newcommand{\id}{\mathbbm{1}}
\newcommand{\inv}{^{-1}}

\newcommand{\mrm}{\mathrm}

\newcommand{\ind}{{\textup{ind}}}

\newcommand{\NN}{\mathbb{N}}

\newcommand{\QQ}{\mathbb{Q}}

\newcommand{\SK}{\chi(S)}
\newcommand{\RK}{\chi(R)}

\newcommand{\tif}{\textup{if }}

\newcommand{\triv}{{\textup{triv}}}

\newcommand{\ud}[2]{\underbracket{#2}_{\mathclap{#1}}}

\newcommand{\udr}[2]{\underbracket{#2}_{\mathrlap{#1}}}

\newcommand{\ZZ}{\mathbb{Z}}

\renewcommand{\~}[1]{\widetilde{#1}}
\renewcommand{\=}[1]{\overline{#1}}

\usepackage[shortlabels]{enumitem}
\newenvironment{enu}{\begin{enumerate}}{\end{enumerate}}
\newenvironment{enua}{\begin{enumerate}[label=\textup{(\alph*)}]
}{\end{enumerate}}
\newenvironment{enuA}{\begin{enumerate}[label=\textup{(\Alph*)}]
}{\end{enumerate}}
\newenvironment{eq}{\begin{equation}}{\end{equation}}

\theoremstyle{definition}

\newtheorem{Def}{Definition}[subsection]

\newtheorem{exa}[Def]{Example}
\newtheorem{rmk}[Def]{Remark}

\theoremstyle{plain}
\newtheorem{prop}[Def]{Proposition}
\newtheorem{thm}[Def]{Theorem}
\newtheorem{mainthm}{Theorem}

\newtheorem{theorem}[Def]{Theorem}

\newtheorem{lem}[Def]{Lemma}
\newtheorem{cor}[Def]{Corollary}

\begin{document}

\title{Quantum wreath products and Schur--Weyl duality II}

\begin{abstract}
In the first part of this series, the authors introduced the quantum wreath product, providing a unified framework that encompasses numerous results previously addressed only through case-by-case analysis.
This paper shifts focus to the fundamental construction of modules over these products, termed wreath modules. 
Our approach utilizes parabolic induction on tensor products combined with a sophisticated labeling scheme based on multipartitions. 
While the underlying constructions are technically involved, they offer a transparent realization of several prominent module families. 
Specifically, these wreath modules recover and unify:
Simple modules over the Ariki--Koike algebra;
Specht and simple modules over the Hu algebra;
(anti)spherical modules and Kashiwara--Miwa--Stern modules over the affine Hecke algebra and its pro-$p$ Iwahori variants.
Finally, we demonstrate that these wreath modules for the Hu algebra serve as a critical component in solving the Ginzburg--Guay--Opdam--Rouquier problem. 
This solution enables a concrete realization of Category O for the rational Cherednik algebra in Type D.
\end{abstract}

\author{\sc Chun-Ju Lai}
\address{Institute of Mathematics \\ Academia Sinica \\ Taipei 106319, Taiwan} 
\email{cjlai@gate.sinica.edu.tw}
\thanks{Research of the first author was supported in part by
MSTC grants 114-2628-M-001-001}

\author{\sc Daniel K. Nakano}
\address{Department of Mathematics\\ University of Georgia \\
Athens\\ GA~30602, USA}
\thanks{Research of the second author was supported in part by
NSF grant DMS-2401184}
\email{nakano@math.uga.edu}

\author{\sc Ziqing Xiang}
\address{ Department of Mathematics\\ Southern University of Science and Technology \\Shenzhen 518055, China} 
\thanks{Research of the third author was supported in part by
NSFC (12350710787) and NSFC (12471311)}
\email{xiangzq@sustech.edu.cn}

\maketitle


\section{Introduction}

\subsection{Quantum Wreath Products}
In \cite{LNX24}, the authors introduced the notion of {\em quantum wreath products}, which enables a simultaneous study of the structure and representation theory for  algebras defined via Bernstein-Lusztig presentations.  
Such a unified approach reveals potential applications to modular representation theory, categorification, and geometric representation theory \cite{MM22,HL24,LM25}.
Similar notions have already been studied intensively by other groups for different purposes, see the work \cite{S20,RS20,BSW20,SS21,BSW21} on Heisenberg categorification, \cite{E22,EG} on Khovanov--Lauda--Rouquier algebras, and \cite{MS} regarding Rees algebras arising from Frobenius algebras.

Notable examples of quantum wreath products include:
\begin{enuA}
\item {Deformations of the group algebra of $G \wr \Sigma_d$ where  $\Sigma_d$ is the symmetric group on $d$ letters.
In particular, When $G = C_m$ is the cyclic group of order $m$, the (modified) Ariki--Koike algebras \cite{AK94,SS05} and the Yokonuma--Hecke algebras \cite{Y67} are all quantum wreath products.}

\item {Affine Hecke algebras for the general linear group $\GL_d$ and their variants (see \cite{LM25,BL26}).
In particular, this include the Iwahori--Hecke algebras, the pro-$p$-Iwahori--Hecke algebras \cite{V16,CS16},
the metaplectic pro-$p$-Iwahori--Hecke algebras \cite{GGK24}, and their degenerate versions.
These algebras are deformed from $\ZZ \wr \Sigma_d$, $(C_m\rtimes\ZZ)\wr \Sigma_d$, $\NN \wr \Sigma_d$, and $(C_m \rtimes \NN) \wr \Sigma_d$, respectively.}
\item {The Hu algebras \cite{Hu02}, which are highly non-trivial quantizations of $\Sigma_m \wr \Sigma_d$ (a group that is not even a complex reflection group) whose representation theory governs that for the complex reflection group $G(d,d,dm)$.
Specifically, when $d=2$, the Hu algebra governs the representation theory of the Hecke algebra of type D.
}
\item {Kleshchev--Muth's  affine zigzag algebras $Z^{\textup{aff}}_{\Gamma}$ for a Dynkin quiver of type ADE \cite{KM19}, which are Morita equivalent to the imaginary strata of affine Khovanov--Lauda--Rouquier algebras.}
\end{enuA}
Given a unital associative algebra $B$ (which we call the base algebra) over a commutative ring $K$, and a choice of parameters $(S, R,\sigma, \rho)$ such that $S, R \in B\otimes B$ and $\sigma, \rho \in \End_K(B\otimes B)$, we denote  by $B\wr \cH(d)$ the quantum wreath product, which is a unital associative algebra generated by elements in $B^{\otimes d}$ and symbols $H_1, \dots, H_{d-1}$ subject to the braid relations, as well as other relations determined by the parameters.

While the braid relations imply that elements of the form $H_w$ for $w\in \Sigma_d$ are well-defined,
the free $K$-submodule $\cH(d)$ with basis $\{ H_w\}_{w\in \Sigma_d}$ is generally not a subalgebra of $B \wr \cH(d)$ because multiplication between $H_w$'s is not closed.
In fact, the theory of quantum wreath product is devoted to handling the failure of $\cH(d)$ to be a subalgebra.

In  \cite{LNX24}, necessary and sufficient conditions on $Q$  so that $B\wr \cH(d)$ affords PBW bases (see Conditions \eqref{def:P1}--\eqref{def:P9} in Section \ref{sec:QWP}) were obtained.
Additionally, a uniform approach to the construction of Schur functors was developed for certain finite-dimensional quantum wreath products.

\subsection{Wreath Modules}

{We begin with an important observation on the simple $B\wr\cH(d)$-modules.
In all known cases, every simple $B\wr\cH(d)$-module is obtained from the tensor product of certain ``building blocks'' via a parabolic induction.
These building blocks -- called the Specht modules, the cuspidal modules, discrete series, or the 1-dimensional Steinberg representations in various contexts -- all have underlying spaces of the form
$M^{\otimes d}\otimes N$,
where $M$ is a $B$-module and $N$ has a module structure over a certain algebra which may depends on $M$.
}

{
The main contribution of this paper is to establish a close connection between the representation theory of $B\wr \cH(d)$ and that of the base algebra $B$.
Given a $B$-module $M$ and a free $K$-module $N$,
we provide a unifying construction of a $B\wr \cH(d)$-module $M\wr N$ with underlying space $M^{\otimes d}\otimes N$. We call them the {\em wreath modules}, which recovers all the aforementioned building block modules.
Motivating examples of our construction, for algebras over $\CC$, are given in Table \ref{tab:WM} below.
Note that every $M$ appearing therein is a simple $B$-module.
}
\begin{table}[!htbp]
{
\[
\begin{array}{|c||c|c|c|}
\hline
\textup{Quantum wreath product} 
    & B 
    &\textup{Wreath modules}
    &\textup{Index set}
\\
\hline
\begin{array}{c}
\textup{Hu algebra}
       \\
     \cA(m)  \cong B\wr \cH(2)
\end{array}
&
\cH_q(\Sigma_m)
&
\begin{array}{l}
(S^\lambda)^{\otimes 2} \otimes \CC_{\sqrt{f_\lambda}} =   S^\lambda \wr S^{(2)}_{\chi_\lambda} 
    \\
(S^\lambda)^{\otimes 2} \otimes \CC_{-\sqrt{f_\lambda}} =   S^\lambda \wr S^{(1,1)}_{\chi_\lambda}     
\end{array}
&
\{\lambda \vdash m\}
\\
\hline
\begin{array}{c}
\textup{Yokonuma--Hecke algebra}
       \\
    Y_{m,d}  \cong B\wr \cH(d)
\end{array}
&
\frac{\CC[t]}{(t^m-1)}
&
(\CC_\lambda)^{\otimes k} \otimes S^\nu
=
\CC_\lambda \wr S^\nu_{\chi_\lambda}
&
\begin{array}{c}
\{\lambda \in \CC|\lambda^m=1\}     
     \\
     (\nu \vdash k\leq d)
\end{array}
\\
\hline
\begin{array}{c}
\textup{Affine Hecke algebra}
       \\
    \cH_q^{\textup{ext}}(\Sigma_d)  \cong B\wr \cH(d)
\end{array}
&
\CC[x^{\pm1}]
&
(\bigotimes_{i=1}^k \CC_{\lambda q^{i-1}}) \otimes \triv
= \CC_{\lambda} \wr S^{(k)}_{\chi_\lambda}
&
\{\lambda \in \CC^\times\}
\\
\hline
\end{array}
\]
\caption{Motivating examples of wreath modules.}
}
\label{tab:WM}
\end{table}
{
Here, $\sqrt{f_\lambda} \in \CC$ is defined in Section \ref{sec:RTHu}. The modules $\CC_{\pm \sqrt{f_\lambda}}$ should be regarded as the ``trivial module'' $S^{(2)}_{\chi_\lambda}$ and the``sign module'' $S^{(1,1)}_{\chi_\lambda}$ over an intermediate algebra $\CC[h]/(h^2 - f_\lambda)$, respectively. 
The subscript $\chi_\lambda$ indicates the dependence of $\lambda$, and will be explained shortly.
}

Now, we develop a new theory determining the conditions under which the $K$-module $M^{\otimes d} \otimes N$ affords a $B \wr \cH(d)$-module structure.
Let $\alpha_M: B^{\otimes d}\otimes M^{\otimes d} \to M^{\otimes d}$ be the structure map of the $B^{\otimes d}$-module $M^{\otimes d}$.
The module structure of $M\wr N$ will be given by:
\eq
(b H_w) \cdot (m \otimes n) := \alpha_{MN}(b \otimes w \otimes m \otimes n),
\endeq
for $b\in B^{\otimes d}, \ w \in \Sigma_d, \ m\in M^{\otimes d}, \ n \in N$, via the following composition:
\eq\label{def:QWPAction}
\xymatrixcolsep{4.5pc}
\xymatrix{
\alpha_{MN}:B^{\otimes d} \otimes K\Sigma_d \otimes M^{\otimes d}\otimes N
\ar[r]^(.56){{\id\otimes \tau^M \otimes \id}} 
&B^{\otimes d}\otimes M^{\otimes d}\otimes \tcH \otimes N
\ar[r]^(.60){{\alpha_M \otimes \alpha_N}}
& M^{\otimes d}\otimes N,
}
\endeq
where
\begin{enuA}
\item $\tau^M = \tau^M(r): K\Sigma_d \otimes M^{\otimes d} \to M^{\otimes d} \otimes \tcH$ is the structure map (see Definition \ref{def:tauaction}), which depends on two $K$-linear maps $\sigma^M, \rho^M \in \End_K(M\otimes M)$, and a choice map $r$ (see \eqref{def:choicemap}). 
Note that the vector space $K\Sigma_d$ is used to keep track of the $K$-submodule $\{H_w\}_{w\in \Sigma_d}$, and should not be interpreted as a group algebra.
\item $\tcH =\tcH^\chi_d$ is the {\em twisted Hecke algebra} (see Definition \ref{def:Hc}).
Here, $\chi = (S^M, R^M) \subseteq \End_B(M)^2$, records the scalar multiples $\chi(S), \chi(R) \in K$ of the actions $S^M, R^M$  of $S, R$ on $M$.
\item $N$ is a module over $\tcH$ with structure map $\alpha_N: \tcH\otimes N \to N$.
\end{enuA}


We show that  \eqref{def:QWPAction} defines a $B\wr\cH(d)$-action if and only if $(M, \tau^M, \tcH)$ is a {\em \based module} (see Definition \ref{def:basedmod})  and with $N \in \tcH$-mod.
The above requirements are equivalent to the following explicit conditions on the tuple $(M, \sigma^M, \rho^M, S^M, R^M)$:
\begin{mainthm}[Theorems \ref{thm:basedmod} and \ref{thm:MwrN}]\label{thmA}
Suppose that $M$ is a $B$-module such that $S$ (resp. $R$) acts on $M\otimes M$ by $\chi(S)$ (resp. $\chi(R)$) in $K$,
and $N$ is an $\tcH^\chi_d$-module.
Then, the $K$-module $M \wr N$ is a $B \wr \cH(d)$-module if and only if the following hold, for all $b\in B^{\otimes d}, m \in M^{\otimes d}, 1\leq i < d$:
\begin{itemize}
\item[\textup{(M2)}]
$\sigma^M_i(b\cdot m) = \sigma_i(b)\cdot \sigma^M_i(m)$,
$\rho_i^M(b\cdot m) =  \sigma_i(b) \cdot \rho^M_i(m) + \rho_i(b) \cdot m$,
\item[\textup{(M4)}]
$\SK(\sigma^M_i)^2+\rho^M_i\sigma^M_i+\sigma^M_i\rho^M_i = l^M_{S_i}\sigma^M_i$,
$\RK (\sigma^M_i)^2+(\rho^M_i)^2 = l^M_{S_i} \rho^M_i + l^M_{R_i}$,
\end{itemize}
where $l^M_X\in \End_K(M^{\otimes d})$ means  the left $B^{\otimes d}$-action by $X \in {B^{\otimes d}}$.
The next three conditions hold additionally for $|i-j|=1$,  if $d \geq 3$:
\begin{itemize}
\item[\textup{(M5)}]
$\sigma^M_i\sigma^M_j\sigma^M_i = \sigma^M_j \sigma^M_i \sigma^M_j$,
$\rho^M_i\sigma^M_j\sigma^M_i = \sigma^M_j \sigma^M_i \rho^M_j$,
\item[\textup{(M6)}]
$\rho^M_i\sigma^M_j\rho^M_i = \SK \sigma^M_j \rho^M_i \sigma_j^M + \rho_j^M\rho_i^M\sigma_j^M + \sigma_j^M \rho_i^M \rho_j^M$,
\item[\textup{(M7)}]
$\rho_i^M\rho_j^M\rho_i^M +\RK \sigma_i^M \rho_j^M \sigma_i^M
= \rho_j^M\rho_i^M\rho_j^M + \RK \sigma_j^M \rho_i^M \sigma_j^M$.
\end{itemize}
\end{mainthm}

\rmk
Recall that Elias' Hecke-type category is a certain monoidal category described in terms of Khovanov--Lauda calculus, in the sense that the generators are given by crossings and boxes in various colors.
It was shown in \cite{E22} that a given morphism space of such a category has a PBW basis if and only if each of the ambiguities in the symmetric category and in the following is resolvable, for each possible coloring:
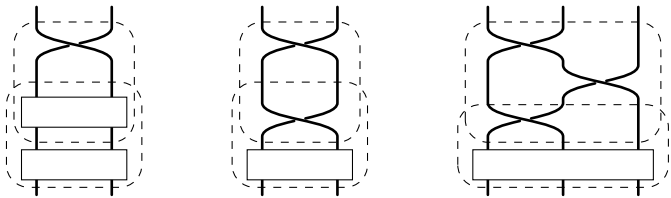
\eq\label{pic:Elias}
  \begin{tikzpicture}
  \pic[
    line width=1pt,
    braid/.cd,
    strand 3/.style={white},
    strand 6/.style={white},
    strand 10/.style={white},
    strand 11/.style={white},
    line cap=round,
    width=1cm,
    crossing height=.5cm,
  ]  at (0,0) {braid={
  s_1-s_4-s_7
  s_8
  s_4-s_7
  s_{10}
  }};
  \draw[black, fill = white] (-0.2,-1.2) rectangle (1.2, -1.6);
  \draw[black, fill = white] (-0.2,-1.9) rectangle (1.2, -2.3);
  \draw[black, dashed, rounded corners = 3mm] (-0.4,-2.4) rectangle (1.4, -1);
  \draw[black, dashed, rounded corners = 3mm] (-0.3,-1.8) rectangle (1.3, -.2);
  \draw[black, fill = white] (2.8, -1.9) rectangle (4.2, -2.3);
  \draw[black, dashed, rounded corners = 3mm] (2.6,-2.4) rectangle (4.4, -1);
  \draw[black, dashed, rounded corners = 3mm] (2.7,-1.8) rectangle (4.3, -.2);
  \draw[black, fill = white] (5.8,-1.9) rectangle (8.2, -2.3);
  \draw[black, dashed, rounded corners = 3mm] (5.6,-2.4) rectangle (8.4, -1.3);
  \draw[black, dashed, rounded corners = 3mm] (5.7,-1.8) rectangle (8.3, -.2);
\end{tikzpicture}
\endeq
In our convention, the resolvability of diagrams \eqref{pic:Elias} not in the symmetric category are equivalent to that \eqref{def:P2}, \eqref{def:P4}, and \eqref{def:P5}--\eqref{def:P7} hold, respectively.

If one introduces additional ingredient (i.e.,  the  shaded boxes below) to the Khovanov--Lauda calculus which correspond to elements in wreath modules,
then conditions \eqref{def:M2}, \eqref{def:M4}, and \eqref{def:taubr1}--\eqref{def:taubr3} correspond respectively to the resolvability of the following diagrams, for each possible coloring:
\eq\label{pic:basedmod}
  \begin{tikzpicture}
  \pic[
    line width=1pt,
    braid/.cd,
    strand 3/.style={white},
    strand 6/.style={white},
    strand 10/.style={white},
    strand 11/.style={white},
    line cap=round,
    width=1cm,
    crossing height=.5cm,
  ]  at (0,0) {braid={
  s_1-s_4-s_7
  s_8
  s_4-s_7
  s_{10}
  }};
  \draw[black, fill = white] (-0.2,-1.2) rectangle (1.2, -1.6);
  \draw[black, fill = gray] (-0.2,-1.9) rectangle (1.2, -2.3);
  \draw[black, dashed, rounded corners = 3mm] (-0.4,-2.4) rectangle (1.4, -1);
  \draw[black, dashed, rounded corners = 3mm] (-0.3,-1.8) rectangle (1.3, -.2);
  \draw[black, fill = gray] (2.8, -1.9) rectangle (4.2, -2.3);
  \draw[black, dashed, rounded corners = 3mm] (2.6,-2.4) rectangle (4.4, -1);
  \draw[black, dashed, rounded corners = 3mm] (2.7,-1.8) rectangle (4.3, -.2);
  \draw[black, fill = gray] (5.8,-1.9) rectangle (8.2, -2.3);
  \draw[black, dashed, rounded corners = 3mm] (5.6,-2.4) rectangle (8.4, -1.3);
  \draw[black, dashed, rounded corners = 3mm] (5.7,-1.8) rectangle (8.3, -.2);
\end{tikzpicture}
\endeq
While it is easy to derive the conditions \eqref{def:M2}--\eqref{def:taubr3} from the resolvability of diagrams, proving the opposite direction is highly-nontrivial.
The proof requires a more involved version of the grand loop argument in \cite{LNX24} in which implications of statements are done by a careful verification of commutativity of diagrams (see Section \ref{sec:app} for details).
\endrmk

\subsection{Applications of the Wreath Module Construction}
\subsubsection{Hu algebras and the Ginzburg--Guay--Opdam--Rouquier Conjecture}
{For the (generalized) Hu algebra $B \wr \cH(d)$ with $B = \cH_q(\Sigma_m)$ and $d \geq 2$,
the wreath modules have been used to construct the corresponding Specht and dual Specht modules, 
whose homological properties play a crucial role in the construction of 1-faithful quasi-hereditary cover.
Such a construction is available is an earlier version of this paper, which is now split into the current version and the third paper \cite{LNX26} with a focus on solving the conjecture on rational Cherednik algebras by Ginzburg, Guay, Opdam and Rouquier \cite{GGOR03}.
}
\subsubsection{Kashiwara--Miwa--Stern Tensor Spaces and (Anti)spherical Modules}
{
For the (skew) polynomial quantum wreath products $B \wr \cH(d)$ where $B$ is of the following form, for some $K$-algebra $F$:
\[
F[x], 
\quad
F[x^{\pm1}],
\quad
F \rtimes K[x],
\quad
\textup{and}
\quad
F \rtimes K[x^{\pm1}],
\]
a special representation that generalizes the Kashiwara--Miwa--Stern (KMS) tensor space \cite{KMS95} has been constructed in \cite{LM25,BL26}, using parabolic inductions and wreath modules of the form $M \wr S^{(d)}$ or $M \wr S^{(1^d)}$. 
Consequently, this quantum wreath product construction leads to the (new) notion of the spherical and antispherical modules.
In particular, these modules are constructed for the metaplectic pro-$p$ Iwahori--Hecke algebras \cite{V16,GGK24} for the $p$-adic group $\GL_d(F)$ over a non-Archimedian local field $F$, for the first time.

We remark that a similar analysis applies to the affine zigzag algebras \cite{KM19} and
certain affine Frobenius Hecke algebras due to Rosso and Savage \cite{RS20}.
We expect that our explicit realization will be useful in obtaining information regarding the corresponding Whittaker models, see \cite{BL26}.
}
\subsubsection{Construction of Simple Modules}
{
Let $B \wr \cH(d)$ be a finite dimensional quantum wreath product which is split semisimple over $K$. 
For all the quantum wreath product that we are aware of (see Section \ref{sec:WMQWP} for details), we observe that any simple $B\wr \cH(d)$-module can be realized as a parabolic induced module from a tensor product of wreath modules.
Moreover, the complete set of non-isomorphic simple $B\wr\cH(d)$-modules is indexed by the following set  of set-indexed multipartitions of size $d$:
\begin{equation}    
I_{B\wr \cH(d)}:= \left\{ \bblambda: I_B \to \Pi 
~\middle|~ 
\sum\nolimits_{\lambda \in I_B} 
\#\bblambda(\lambda) = d
\right\},
\end{equation}
where $\Pi$ is the set of all partitions and the set of non-isomorphic simple $B$-modules is $\{L_\lambda ~|~ \lambda \in I_B\}$.

Now, given $\bblambda \in I_{B \wr \cH(d)}$, we describe the construction of the simple module $L(\bblambda)$:
\begin{itemize}
\item Step 1: Suppose that the support of $\bblambda$ is $\{\lambda^{(1)}, \dots, \lambda^{(r)}\} \subseteq I_B$. 
This means that we will build $L(\bblambda)$ from the wreath modules
$L_{\lambda^{(i)}}\wr N_i$
for some module $N_i$ to be described.

\item Step 2: For all $i$, write $\nu_i := \bblambda(\lambda^{(i)})$.
Note that $\nu_i$ is a partition of some integer $\mu_i \leq d$.
We then set $N_i = S^{\nu_i}_{\chi_i}$ to be the Specht module (see Section \ref{sec:tHASpecht}) over the twisted Hecke algebra $\tcH^{\chi_i}_{\mu_i}$ with relations determined by $\chi_i$.

\item Step 3: Since $\mu:=(\mu_1, \dots, \mu_r) \vDash d$ is a composition,
we may define a parabolic subalgebra $B\wr \cH(\mu) := (B\wr \cH(\mu_1))\otimes \dots \otimes (B\wr \cH(\mu_r))$ of $B\wr\cH(d)$.
Thus, the outer tensor product
$(L_{\lambda^{(1)}}\wr N_1) \boxtimes \dots \boxtimes (L_{\lambda^{(r)}}\wr N_r)
\in B\wr \cH(\mu)\textup{-mod}$
induces up to a $B\wr \cH(d)$-module  
\[
L(\bblambda) := \ind_{B\wr \cH(\mu)}^{B \wr \cH(d)}\big(
(L_{\lambda^{(1)}}\wr N_1) \boxtimes \dots \boxtimes (L_{\lambda^{(r)}}\wr N_r)
\big),
\]
which is the desired simple module corresponding to $\bblambda$.
Observe that one can construct the tensor product using different orders of elements in supp$(\bblambda)$, but the induced module are all isomorphic.
\end{itemize}



The case of the simple modules for the Hu algebras, the modified Ariki--Koike algebras, and 
the Yokonuma--Hecke algebras will be elaborated on in Section \ref{sec:WMQWP}.
We will also discuss the infinite dimensional counterparts for realizing simple modules of the (degenerate) affine Hecke algebras in Section \ref{sec:cuspidal}.}
\section{Basic Definitions and Motivating Examples}\label{sec:Wreath}
\subsection{Quantum Wreath Products} \label{sec:QWP}
Throughout this paper, let $K$ be a commutative ring. All algebras  are free modules over $K$, and all tensor products are considered to be over $K$, unless specified otherwise.
Let $B$ be an associative algebra, and let $d\geq 2$ be an integer.
By a choice of parameters, we mean a choice of a quadruple $Q = (R,S,\rho,\sigma)$ where $R, S \in B\otimes B$, $\rho, \sigma \in  \End_K(B\otimes B)$ such that $\sigma$ is an automorphism.

For $Z \in B\otimes B$ and $\phi \in \End_K(B\otimes B)$, define $Z_i \in B^{\otimes d}$ and $\phi_i \in \End_K(B^{\otimes d})$, respectively, by 
\eq\label{def:Xphii}
Z_i :=  1^{\otimes i-1} \otimes Z \otimes 1^{\otimes d-i-1},
\quad
\phi_i(b_1 \otimes \dots \otimes b_d) := b_1 \otimes \dots \otimes b_{i-1}\otimes  \phi(b_{i} \otimes b_{i+1}) \otimes b_{i+2} \otimes \dots \otimes b_{d}.
\endeq 
\Def
The {\em quantum wreath product} {with parameter $Q$} is the associative $K$-algebra generated by the algebra $B^{\otimes d}$ and $H_1, \dots, H_{d-1}$ such that the following conditions hold: for $1\leq k \leq d-2, 1\leq i \leq d-1, |j-i|\geq 2$,
\begin{align}
&\textup{(braid relations) } \label{def:BR}
& H_k H_{k+ 1} H_k = H_{k+ 1} H_k H_{k + 1},  \quad H_i H_j = H_j H_i,
\\
&\textup{(quadratic relations) }\label{def:QR}
&H_i^2 = {S_i}H_i + {R_i},
\\
&\label{def:WR}\textup{(wreath relations) }
& H_ib = \sigma_i(b)H_i + \rho_i(b) \quad ( b\in B^{\otimes d}).
\end{align}
\noindent Often times, we refer this algebra  as $B \wr_Q \cH(d)$, or $B \wr \cH(d)$ whenever it is convenient.
\endDef

A $K$-basis of $B \wr \cH(d)$ is called a {\em PBW} basis if it is of one of the following forms:
\eq\label{def:PBWbases}
\{ (b_{i_1} \otimes \dots\otimes b_{i_d}) H_w ~|~ i_j \in I, w\in \Sigma_d\},
\quad\textup{and}\quad
\{H_w (b_{i_1} \otimes \dots\otimes b_{i_d})  ~|~ i_j \in I, w\in \Sigma_d\},
\endeq
where $\{b_i ~|~ i\in I\}$ is a $K$-basis of $B$, and $H_w := H_{i_1} \dots H_{i_N}$ for a reduced expression $w = s_{i_1} \dots s_{i_N} \in \Sigma_d$. 
This element $H_w$ is well-defined since the braid relations hold.
It is proved in \cite[Theorem 3.3.1]{LNX24} that the existence of a PBW basis implies the existence of another PBW basis with reversed order.

In this paper, we make the following simplifying assumption:
\begin{align}
\label{asm:basis}&B\wr \cH(d) \textup{ has a PBW basis in the sense of \eqref{def:PBWbases}}.
\end{align}
The statement \eqref{asm:basis} above is equivalent (cf. \cite[Theorem 3.3.1]{LNX24}) to 
\[
\textup{Conditions }\eqref{def:P1} \textup{ -- } \eqref{def:P4}\textup{, and Conditions }\eqref{def:P5}\textup{ -- } \eqref{def:P9}\textup{ additionally if }d\geq 3.
\]
These conditions are given as follows:
\begin{align}
\tag{P1}&\label{def:P1}\sigma(1\otimes 1) = 1\otimes 1, 
\quad \rho(1 \otimes1)=0,
\\
\tag{P2}&\label{def:P2} \sigma(ab) = \sigma(a)\sigma(b),
\quad
\rho(ab) =  \sigma(a)\rho(b) + \rho(a)b,
\\
\tag{P3}&\label{def:P3}
\sigma(S)S + \rho(S) + \sigma(R) = S^2+R,
\quad
\rho(R) + \sigma(S)R = SR,
\\
\tag{P4}&\label{def:P4}
r_S\sigma^2+\rho\sigma+\sigma\rho = {l}_{S}\sigma,
\quad
r_R\sigma^2+\rho^2 = {l}_S\rho + {l}_R,
\end{align}
where  ${l_X}, r_X$ for $X\in B\otimes B$ are $K$-endomorphisms defined by  {left and} right multiplication in $B\otimes B$ by $X$, respectively. {The next five conditions will only be necessary when $d\geq 3$:}
\begin{align}
\tag{P5}&\label{def:P5}\sigma_i\sigma_j\sigma_i = \sigma_j \sigma_i \sigma_j,
\quad \rho_i\sigma_j\sigma_i = \sigma_j \sigma_i \rho_j,
\\
\tag{P6}&\label{def:P6}
\rho_i\sigma_j\rho_i = r_{S_j} \sigma_j \rho_i \sigma_j + \rho_j \rho_i \sigma_j + \sigma_j \rho_i \rho_j,
\\
\tag{P7}&\label{def:P7}\rho_i\rho_j\rho_i + r_{R_i} \sigma_i \rho_j \sigma_i
= \rho_j\rho_i\rho_j + r_{R_j} \sigma_j \rho_i \sigma_j,
\end{align}
where $\{i,j\}= \{1,2\}$, $r_X$ for $X\in B^{\otimes 3}$ is understood as right multiplication in $B^{\otimes 3}$ by $X$. 
\begin{align}
\tag{P8}&\label{def:P8}&S_i = \sigma_j\sigma_i(S_j), \quad R_i = \sigma_j\sigma_i(R_j),
\quad
\rho_j\sigma_i(S_j) = 0 = \rho_j\sigma_i(R_j),
\\
\tag{P9}&\label{def:P9}& \sigma_j\rho_i(S_j) S_j +\rho_j\rho_i(S_j) + \sigma_j\rho_i(R_j) = 0 =  \rho_j\rho_i(R_j) + \sigma_j\rho_i(S_j) R_j,
\end{align}
where $\{i,j\}= \{1,2\}$. 



\subsection{Parabolic Subalgebras}

For the purposes of this paper, it will be useful to define parabolic subalgebras for quantum wreath products.  

\Def\label{def:Hpar}

For a composition $\mu \vDash d$, denote the corresponding {\em parabolic subalgebra} of $B \wr \cH(d)$ by
\[
B \wr \cH(\mu) := \< B^{\otimes d}, H_w~|~ w \in \Sigma_\mu\> \subseteq B \wr \cH(d).
\]
The tensor product $(B \wr \cH(\mu_1)) \otimes \dots \otimes (B\wr \cH(\mu_r))$ among quantum wreath products of the same parameter $Q$ is canonically isomorphic to the parabolic subalgebra $B\wr \cH(\mu)$ via $(b_1 H_{w_1})\otimes \dots \otimes (b_r H_{w_r}) \mapsto (b_1\otimes \dots \otimes b_r) H_{w_1 \dots w_r}$ where we used the identification $\Sigma_{\mu_1}\times \dots \times \Sigma_{\mu_r} \cong \Sigma_\mu$.

\subsection{Hu Algebras as Quantum Wreath Products}
In this subsection, assume that $K$ is a field of characteristic not equal to 2 containing {$v$ and} $q = v^2$.
Let $B = \cH_q(\Sigma_m)$ be the Hecke algebra for $\Sigma_m$ with generators $I_1, \dots I_{m-1}$ subject to the braid relations and the following normalized quadratic relations:
\eq
(I_i + v\inv)(I_i - v) = 0,
\quad
\textup{or {equivalently},}
\quad
I_i^2 = (v-v\inv) I_i +1.
\endeq
In particular, the bar-involution on $B$ sends each $I_{w}$ to $\overline{I}_w := I_{w\inv}\inv$.
Let $\iota$ be the identification $\cH_q(\Sigma_m \times \Sigma_m) \equiv B\otimes B$.
\begin{prop} [{\cite[Proposition~6.2.2, Theorem 6.3.3]{LNX24}}]\label{prop:H1}
Suppose that  $f_{2m}(q):= \prod_{i=0}^{2m-1}(1+q^i)$ is nonzero.
Then, $Z_m := \iota(H_1^2) \in Z(B \otimes B)$, where 
\eq
H_1 = v^{m(2m-1)}{\textstyle \sum_{\epsilon = (\epsilon_1, \dots, \epsilon_m) \in \{+,-\}^m}  
I_{m\to 2m-1}^{\epsilon_1}
\dots
I_{1\to m}^{\epsilon_m} C_\epsilon},
\endeq
where, for integers $a < b$: 
\eq
I_{a\to b}^{\epsilon_i} := 
\begin{cases}
I_a I_{a+1} \dots I_b &\tif \epsilon_i = +;
\\
I\inv_a I\inv_{a+1}\dots I\inv_b &\tif \epsilon_i = -,
\end{cases}
\quad
C_\epsilon = \prod_{i;\epsilon_i=-}  I_{2m-i} \dots  I_{m+1} I_{m+1 \dots}I_{2m-i}.
\endeq
\end{prop}
\exa
For $m=1$, the central element is $Z_1 = (v(I_1 + I_1\inv))^2 = v^2 (v+v\inv)^2 = (q+1)^2$. That is, a square of quantum two. For $m=2$,
\eq
 Z_2=v^{12}((v-v\inv)^2f(v) I_{1}\otimes I_1
+(v-v\inv)( f(v) + 2(v^2 + v^{-2}))(I_1\otimes 1 +1\otimes I_1)
+2f(v)),
\endeq
where $f(v) := v^4+2v^2-2+2v^{-2}+v^{-4}$.
While $Z_m$ does not seem to have a simple closed formula, we will see that the $Z_m$-action on any pair $S_m^\lambda \otimes S_m^\lambda$ of Specht modules over $\cH_q(\Sigma_m)$ by an explicit scalar multiplication.
\endexa

The Hu algebra is the quantum wreath product $A_q(m) = B\wr \cH(2)$ with {$B = \cH_q(\Sigma_m)$ and} the following choices {of parameters}:
\eq\label{par:Hu}
S=0,
\quad
R = Z_m,
\quad 
\sigma: a\otimes b \mapsto b\otimes a,
\quad
\rho = 0.
\endeq
In other words, $A_q(m)$ is generated by $B^{\otimes 2}$ and $H_1$ subject to the following relation:
\eq
H_1^2 = Z_m,
\quad H_1(a \otimes b) = (b \otimes a) H_1,
\qquad
(a,b \in B).
\endeq
\subsection{Representation Theory of Hu Algebras}\label{sec:RTHu}
Consider the Hu algebra  $A_q(m)$. For $\lambda\vdash m$, denote by $S_m^\lambda$ the Specht module over $\cH_q(\Sigma_m)$, and by $D_m^\lambda$ the quotient of $S_m^\lambda$ by its radical.
Let $\lambda', \lambda'' \vdash m$. It follows from \cite[Lemmas 3.2 and 3.4]{Hu02} that $Z_m$ acts on $S_m^{\lambda'} \otimes S_m^{\lambda''}$ by multiplication by a scalar which we call $f_{\lambda', \lambda''} \in \QQ(v)$. If, additionally, $\lambda' = \lambda = \lambda''$, then $\sqrt{f_{\lambda}} := \sqrt{f_{\lambda,\lambda}}$ is well-defined {(up to a sign)} and lies in $\ZZ[v, v\inv]$ thanks to \cite[Theorem 3.5]{Hu02}.

The complete set of Specht and simple modules over the Hu algebra is constructed explicitly.
For bipartition $(\lambda', \lambda'')$ of $(m,m)$, define
\eq
\begin{split}
&S(\lambda', \lambda'') := \ind_{B\otimes B}^{A_q(m)}(S_m^{\lambda'} \otimes S_m^{\lambda''}),
\quad 
D(\lambda', \lambda'') := \ind_{B\otimes B}^{A_q(m)}(D_m^{\lambda'} \otimes D_m^{\lambda''}).
\end{split}
\endeq
Note that when $\lambda' = \lambda = \lambda''$, the induced modules split into two summands, i.e.,
\eq
\ind_{B\otimes B}^{A_q(m)}(S_m^{\lambda} \otimes S_m^{\lambda}) = S(\lambda, \lambda)_{+} \oplus S(\lambda, \lambda)_{-},
\quad
\ind_{B\otimes B}^{A_q(m)}(D_m^{\lambda} \otimes D_m^{\lambda}) = D(\lambda, \lambda)_{+} \oplus D(\lambda, \lambda)_{-}.
\endeq
where 
\eq\label{def:Sllpm}
\begin{split}
&S(\lambda, \lambda)_{\pm} := \textup{Span}_K\{\sqrt{f_{\lambda}} (x\otimes y) \pm (x\otimes y)H_1 ~|~x \in S_m^{\lambda'}, \ y\in S_m^{\lambda''}\},
\\
&D(\lambda, \lambda)_{\pm} := \textup{Span}_K\{\sqrt{f_{\lambda}} (x\otimes y) \pm (x\otimes y)H_1 ~|~x \in D_m^{\lambda'}, \ y\in D_m^{\lambda''}\}.
\end{split}
\endeq
\begin{prop}[{\cite[Theorem 4.5]{Hu02}}]\label{prop:RepHu}
Let $A_q(m) = B \wr \cH(2)$ be the Hu algebra. Then,
\begin{enua}
\item
There is a notion of the Specht modules over over $A_q(m)$, whose complete set of isomorphism classes is given by 
\[
\{S(\lambda',\lambda'') ~|~ \lambda' \neq \lambda'' \in \Pi_m\} \sqcup \{S(\lambda,\lambda)_+, S(\lambda,\lambda)_- ~|~ \lambda \in \Pi_m\}.
\]
\item
The complete set of isomorphism classes of simple modules over $A_q(m)$ is given by 
\[
\{D(\lambda',\lambda'') ~|~ \lambda', \lambda'' \in \Pi_m, \ D^{\lambda'}_m \neq 0 \neq D^{\lambda''}_m\}  \sqcup \{D(\lambda,\lambda)_+, D(\lambda,\lambda)_- ~|~ \lambda \in \Pi_m, \ D^{\lambda}_m \neq 0\}.
\]
\end{enua}
\end{prop}

\subsection{A New Viewpoint}\label{sec:newviewptHu}
In this section, we will preview a new setup which better describes the representation theory for the Hu algebras and other quantum wreath products.

\subsubsection{A Multipartition Approach}

Recall that $\Pi_m = \{ \lambda = (\lambda_1, \lambda_2, \dots) \vdash m\}$ is the set of partitions of $m$.
Let $\Pi := \bigcup_{m \geq 0} \Pi_m$ be the set of all ({possibly} empty) partitions. 
A multipartition $(\lambda^{(1)}, \dots, \lambda^{(n)})$ of $n$-parts can be identified as a map $\{1, \dots, n\} \to \Pi$, $i \mapsto \lambda^{(i)}$.

For our purposes, we need to label the Specht and simple modules with a new index set $\Omega$ consisting of ``multipartitions'' of the form $\bblambda: \Pi_m \to \Pi$ such that the sizes of all $\bblambda(\lambda)$, for $\lambda \in \Pi_m$, add up to $d$. 
It is sometimes useful to express these multipartitions $\bblambda$  as a sum of characteristic functions via
$\bblambda \equiv \sum_{\lambda \vdash m} e_\lambda^{\bblambda(\lambda)}$ where
\eq\label{def:charfcn}
e_\lambda^\nu(\lambda) := \nu,
\quad
e_\lambda^\nu(\gamma) := \varnothing 
\quad
\textup{for all }
\quad \gamma \in \Pi_m \setminus \{\lambda \}.
\endeq
For example, if $m = d = 2$, then
\[
\Pi_m = \{(2), (1,1)\}, 
\quad
\Omega = \{ e_{(2)}^{(2)}, \ e_{(2)}^{(1,1)},\ e_{(1,1)}^{(2)},\ e_{(1,1)}^{(1,1)},\ e_{(2)}^{(1)} + e_{(1,1)}^{(1)}\}.
\] 
\subsubsection{Wreath Modules}
The building blocks of a Specht module over $B \wr \cH(d)$ is the wreath module $M \wr N$.
We will show that $S^\lambda_m\wr N$ affords a module structure over the Hu algebra if and only if $N$ has a module structure over the 
{\em twisted Hecke algebra} $\tcH^\chi_d$  which we will discuss in Section \ref{sec:tHA}, for some twist $\chi = \chi_\lambda$. 
Moreover, $\tcH^\chi_d$ admits a family $\{S_{\chi}^\nu ~|~ \nu \vdash d\}$ of Specht modules.

{Given $\bblambda \in \Omega$, suppose that $\nu$ is in the image of $\bblambda$, i.e., $\nu =\bblambda(\lambda)$ for some $\lambda \vdash m$. This indicates that the wreath module $S^\lambda_m \wr S_{\chi_\lambda}^\nu$ will be used to construct $S^\bblambda$.}
When $d=1$, it is understood that $\tcH^{\chi}_1 \equiv K$ with a single Specht module $S_{\chi}^{(1)} \equiv K$.
Hence, the wreath module $S^\lambda_m \wr S_{\chi_\lambda}^{(1)} \equiv S^\lambda_m$ is the Specht module over $B$.
\subsubsection{Parabolic Inductions}
Suppose that the support of $\bblambda$ is $\{\lambda^{(1)}, \lambda^{(2)}, \dots, \lambda^{(r)}\}$.
Write $\mu_i := \# \lambda^{(i)}$ {to be the total number of boxes in the partition}, and hence $\mu = (\mu_1, \dots, \mu_r) \vDash d$. 
Since each $S^{\lambda^{(i)}}_m \wr S_{\chi_\lambda}^{\bblambda(\lambda^{(i)})}$ is a module over $B \wr\cH(\mu_i)$,
their tensor product is a module over the parabolic subalgebra $B\wr \cH(\mu)$.

Therefore, for each $\bblambda \in \Omega$ we can construct a Specht module 
\[
S^\bblambda = 
\ind_{\mu}^{2} \big( \bigotimes\nolimits_i (S_{m}^{\lambda^{(i)}} \wr S_{\chi_{\lambda^{(i)}}}^{\bblambda(\lambda^{(i)})}) \big),
\quad
(\textup{for distinct }\lambda^{(i)}\textup{'s}),
\]
where we use the shorthand notation
\eq\label{eq:ind}
\ind_\mu^d =\ind_{\mu_1, \mu_2, \dots, \mu_r}^d := \ind_{B \wr \cH(\mu)}^{B \wr \cH(d)},
\qquad
(\textup{for composition } \mu = (\mu_1, \dots, \mu_r) \vDash d).
\endeq
\subsubsection{An Identification}
We will see in Example \ref{ex:HuClifford} that the Specht modules in Section \ref{sec:RTHu} are identified as in the table below:
\[
\begin{array}{|c|ccc|ccc}
\hline
\textup{Specht module in \cite{Hu02}}&S(\lambda, \lambda)_+ &S(\lambda, \lambda)_- &S(\lambda', \lambda'')
\\
\hline
\textup{Multipartition }\bblambda&  e_\lambda^{(2)} & e_\lambda^{(1,1)} & e_{\lambda'}^{(1)} + e_{\lambda''}^{(1)}
\\
\hline
\textup{Our Specht module }S^\bblambda & S_m^\lambda \wr S^{(2)}_{\chi_\lambda} & S_m^\lambda \wr S^{(1,1)}_{\chi_\lambda} 
& \ind_{1,1}^2\big((S_m^{\lambda'} \wr S_{\chi'}^{(1)})\otimes (S_m^{\lambda''}\wr S_{\chi''}^{(1)}) \big)
\\
\hline
\end{array}
\]
Therefore, we can then paraphrase Proposition~\ref{prop:RepHu} as follows:
\setcounter{Def}{4}
\begin{cor}\label{cor:gHuIrrep}
Let $A_q(m)$ be the Hu algebra.
For $\lambda \vdash m$, let $\chi_\lambda$ be the twist given by  
$\chi_\lambda(S) = 0$, $\chi_\lambda(R) = f_{\lambda,\lambda}$.
Then,
\begin{enua}
\item
There is a notion of the Specht modules over over $A_q(m)$, whose complete set of isomorphism classes is given by 
\[
\{S^\bblambda ~|~ \bblambda \in \Omega\} =
\left\{
S^\lambda_m \wr S_{\chi_\lambda}^{(2)}, \
S^\lambda_m \wr S_{\chi_\lambda}^{(1,1)}, \
\ind_{1,1}^2(S^\lambda_m \otimes S^{\lambda'}_m)
~\middle|~ 
\lambda, \lambda' \in \Pi_m{, \lambda \neq \lambda'}
\right\}.
\]
\item
The complete set of isomorphism classes of simple modules over $A_q(m)$ is given by 
\[
\left\{
\ind_{\mu}^{2} \big( \bigotimes\nolimits_i (D_{m}^{\lambda^{(i)}} \wr S_{\chi_{\lambda^{(i)}}}^{\nu^{(i)}}) \big)
~\middle|~
\mu \vDash 2, 
\  \nu^{(i)} \vdash \mu_i,
\ \lambda^{(i)} \vdash m,
\ D_{m}^{\lambda^{(i)}} \neq 0
\right\},
\]
where $\lambda^{(i)}$'s are distinct partitions.
\end{enua}
\end{cor}


\section{Twisted Hecke Algebras}\label{sec:tHA}
For our constructions, it will be more convenient to describe the  modules in Section \ref{sec:RTHu} and others using a notation we call the wreath modules $M \wr N$, where $M$ is a module over $B\otimes B$ and $N$ is a free $K$-module such that $M \wr N = M^{\otimes d} \otimes N$ as $K$-modules.

While it is straightforward to define an action of $B^{\otimes d}$ on $M^{\otimes d}$ using the $B$-module structure of $M$, it is surprisingly not easy to determine what module structure one needs for $N$ so that $M \wr N$ becomes a module over the quantum wreath product.
We begin with studying a different quantum wreath product with base algebra being an endomorphism algebra.

\subsection{Twisted Hecke Algebras}\label{SS:twisted} 
Let $B \wr \cH(d)$ be a fixed quantum wreath product,
and let $M$ be a fixed left $B$-module.
Then $\End_B(M)$ is a (not necessarily  commutative) $K$-algebra.

\Def[Twisted Hecke algebras]\label{def:Hc}
Let $E \subseteq \End_B(M)$ be a commutative $K$-subalgebra. 
By a {\em twist} we mean a choice of pair of elements
\eq\label{def:twistchi}
\chi = (S^M, R^M) \in (E\otimes E)^2.
\endeq 
Hence, following \eqref{def:Xphii}, $S^M_i, R^M_i$ are elements in $E^{\otimes d}$ for all $i$. 
Define, given $\sigma$ and $\rho$ in $\End_K(E\otimes E)$, the {\em twisted Hecke algebra} $\tcH^\chi_d$ as the quantum wreath product $E \wr \cH(d)$ generated by $E^{\otimes d}$ and Hecke-like elements $h_1, \dots, h_{d-1}$, with $S = S^M$ and $R = R^M$.
In particular, the following holds for all $1\leq i< d$:
\eq\label{def:QRchi}
h_i^2 =S^M_i h_i + R^M_i.
\endeq
\endDef
While $\sigma$ and $\rho$ can be arbitrary for a twisted Hecke algebra, 
in this paper we make the following assumptions:
\eq\label{asm:twistchi}
E = K,
\quad
\sigma = \id, \quad 
\rho = 0,
\quad
S^M(m) = \chi(S) m,
\quad
R^M(m) = \chi(R) m,
\quad
(m \in M\otimes M),
\endeq
for some $\chi(S), \chi(R) \in K$.
In words, we assume that $S^M$ (resp., $R^M$) is a scalar multiplication by an element in $K$ that we call $\chi(S)$ (resp., $\chi(R)$).

Let $h_w := h_{i_1} \dots h_{i_N}$ for a reduced expression $w = s_{i_1} \dots s_{i_N} \in \Sigma_d$. 
\begin{lem}
Retain the assumptions \eqref{asm:twistchi}.
Then, $\{h_w ~|~ w \in \Sigma_d\}$ forms a $K$-basis of the twisted Hecke algebra $\tcH^\chi_d = E \wr \cH(d)$.
\end{lem}
\proof
Since $E=K$, it suffice to show that $\tcH^\chi_d$ has PBW bases. 
It then follows from \cite{LNX24} that one needs to  verify \eqref{def:P1}--\eqref{def:P9}, which reduce to the following:
\eq
\begin{split}
&S^M \circ f = f \circ S^M,
\quad
R^M \circ f = f\circ R^M
\quad
\textup{for all }
f\in E\otimes E,
\\
&
S^M_i = S^M_j,
\quad
R^M_i = R^M_j
\quad
\textup{for all }
1\leq i, j < d.
\end{split}
\endeq
The first line follows from the fact that $f$ is $K$-linear; while the second line follows since $E^{\otimes d}$ is a tensor product over $K$.
\endproof

We make several observations. First, one view $\chi$ as a central character $\chi: Z(B\otimes B) \to K$, but only the values at $S$ and $R$ matter.

Second, it turns out that with the assumptions \eqref{asm:twistchi}, the twisted Hecke algebra is nothing but the generic Hecke algebra of type A.
In this paper, it happens to be the case that $E=K$ in all the examples we consider,
so we do not pursue a more general theory without assuming $E=K$ for now.
For future work which deals with more complicated quantum wreath products, we expect $E \neq K$ in the corresponding representation theory.
\endrmk
\subsection{Trivial and Sign Modules}
In the classic theory of representation theory of Weyl groups and Hecke algebras,
there are two important one-dimensional modules called the trivial and the sign modules.

We remind readers that such a notion of trivial and sign modules over the twisted Hecke algebra $\tcH^\chi_d$ exist if and only if \eqref{def:QRchi} splits into linear factors, i.e., 
\eq\label{def:QRchi2}
h_i^2 - \chi(S) h_i -\chi(R) = 0 = (h_i - r_1)(h_i - r_2 )
\quad
\textup{for some}
\quad
r_1, r_2 \in E^{\otimes d} \equiv K.
\endeq
In such a case, there are one-dimensional $\tcH_d^\chi$-modules, called $K_{r_1}$ and $K_{r_2}$, on which each $h_i$ acts by multiplications by $r_1$ and $r_2$, respectively. 

\begin{Def}
Let $r_1\in K^\times$ be a root of $x^2 - \chi(S) x -\chi(R) \in K[x]$.
Let $r_2$ be the other root.
Call $K_{r_1}$ the {\em trivial module}, and $K_{r_2}$  the {\em sign module}. 
\end{Def} 
\noindent Note the choice of $r_1$ is not canonical in general.    

\exa
If the quadratic relation is $0 = (h_i {-} q)(h_i {+} 1)$ for some $q$ in $K$, e.g., $\tcH^\chi_d$ is the Hecke algebra or group algebra for $\Sigma_d$. 
Then, the conventional choice is to pick $r_1 = q$ and $r_2 = -1$.
 
Additionally, these one-dimensional modules $K_1$ and $K_{-q}$ are referred to as the Specht module corresponding to the partitions $(d)$ and $(1^d)$ of $d$, respectively.
\endexa

\subsection{Specht Modules}\label{sec:tHASpecht}
Suppose that \eqref{def:QRchi2} holds. Equivalently, for all $i$:
\eq\label{eq:eigen}
(h_i-r_1)h_i = r_2 (h_i-r_1),
\quad
(h_i-r_2)h_i = r_1 (h_i-r_2).
\endeq
For any partition $\mu \vdash m$, the {\em Specht module} $S^\mu_\chi$ over $\tcH^\chi_d$ is given by
\eq
S^\mu_\chi := x^\chi_\mu h_{w_\mu}  y^\chi_{\mu^t} \tcH^\chi_d,
\quad
\textup{where}
\quad
x^\chi_\mu := \sum_{w\in\Sigma_\mu} (-r_2)^{\ell(w)} h_w,
\quad
y^\chi_\mu := \sum_{w\in\Sigma_\mu} (-r_1)^{\ell(w)} h_w,
\endeq
and $w_\mu \in \Sigma_m$ is the shortest element that takes the standard tableau $\ft^\mu$ of shape $\mu$ to the transpose of $\ft^{\mu^t}$. The definition of 
$h_{w}$ is given in Section~\ref{SS:twisted}. 

For example,
\eq
\mu = (3,1),
\quad
\ft^\mu =\iyoung{123,4}, 
\quad
(\ft^{\mu^t})^t =\Big(\iyoung{12,3,4} \Big)^t = \iyoung{134,2},
\quad
\textup{and}
\quad
w_\mu = (2 3 4) = s_3 s_2.
\endeq
In particular, $S^{(d)}_\chi$ is the one-dimensional module spanned by $x^\chi_{(d)}$ on which $h_i$'s all act by multiplication by $r_1$, thanks to \eqref{eq:eigen},
and hence $S^{(d)}_\chi \cong K_{r_1}$ is the trivial module.
Similarly, $S^{(1^d)}_\chi \cong K_{r_2}$  is the sign module.

\exa\label{ex:HuClifford}
For the Hu algebra $A_q(m) = B \wr \cH(2)$ with $B = \cH_q(\Sigma_m)$.
Consider the Specht module $M = S^\lambda_m$.
We set $E = \End_B(M) \cong K$.
Let $\chi_\lambda = (S^M, R^M)$ where $S^M$ is the zero map so $\chi(S) = 0$, and $R^M$ is the multiplication by $\chi(R) = f_{\lambda,\lambda} \in K$.
The twisted Hecke algebra for this $\chi$ is 
\eq
\tcH^\chi_2 = E \wr \cH(2) = \langle h_1\rangle/( h_1^2 - f_{\lambda,\lambda}).
\endeq
Since there is an element $\sqrt{f_{\lambda}} \in K$ which squares into $f_{\lambda,\lambda}$, the relation \eqref{def:QRchi2} holds.
We may choose $r_1 = \sqrt{f_{\lambda}}$ and $r_2 = -\sqrt{f_{\lambda}}$ with corresponding one-dimensional modules
\eq
K_{\sqrt{f_{\lambda}}} \equiv S^{(2)}_{\chi_\lambda},
\quad
\textup{and}
\quad
K_{{-\sqrt{f_{\lambda}}}} \equiv S^{(1,1)}_{\chi_\lambda}.
\endeq
A direct calculation (see Section \ref{sec:gHu}) leads to the following identification:
\eq\label{eq:HutauM}
S(\lambda, \lambda)_{+} = S^\lambda_m \otimes S^\lambda_m \otimes S^{(2)}_{\chi_\lambda},
\quad
S(\lambda, \lambda)_{-} = S^\lambda_m \otimes S^\lambda_m \otimes S^{(1,1)}_{\chi_\lambda},
\endeq
where the  $B\wr \cH(2)$-action is given by 
\eq\label{eq:HuAction}
(b_1\otimes b_2) \cdot (m_1 \otimes m_2 \otimes y) = (b_1 \cdot m_1) \otimes (b_2\cdot m_2) \otimes y,
\quad
H_1 \cdot (m_1 \otimes m_2 \otimes y) = m_2 \otimes m_1 \otimes (h_1 \cdot y),
\endeq
for all  $b_i \in B$, $m_i \in M$, $y \in S_\chi^\nu$, and $\nu \vdash 2$.

\subsection{The KMS Action}\label{sec:endo}
From the special case \eqref{eq:HuAction} with the Hu algebras,
it is tempting to guess that the following gives a $B\wr \cH(d)$-action on $M \wr N$:
\eq\label{def:naiveaction}
(bH_w) \cdot (m \otimes y) = \sigma_w(b \cdot m)\otimes (h_w \cdot y)
\quad
\textup{for}
\quad
b\in B^{\otimes d}, \ w \in \Sigma_d, \ m \in M^{\otimes d} , \ y \in N.
\endeq
Here, $\sigma_w := \sigma_{i_1} \dots \sigma_{i_N}$ is $w = s_{i_1} \dots s_{i_N}$ is a reduced expression.
It turns out that \eqref{def:naiveaction} does not comply with the following  motivating example of the wreath module that arises from the tensor space representation $T_n^{\otimes d}$ over the affine Hecke algebra  due to Kashiwara, Miwa and Stern \cite{KMS95}.
\exa\label{ex:AHAClifford}
Let $\widehat{\cH}_q(d) $ be the (extended) affine Hecke algebra for $GL_d$.
It is known that $\widehat{\cH}_q(d) \equiv B \wr \cH(d) $ is a quantum wreath product with the following choices, for $a,b\in B := K[x^{\pm 1}]$:
\eq
\begin{split}
&S = (q-1)(1\otimes 1),
\quad
R = q(1\otimes 1),
\\
&\sigma: a\otimes b \mapsto b\otimes a,
\quad
\rho(a\otimes b) = -(q-1)\frac{a\otimes b - b\otimes a}{x\otimes 1 - 1\otimes x} (1\otimes x).
\end{split}
\endeq
In \cite{KMS95}, they defined a tensor space representation 
\eq
T_{n}^{\otimes d},
\quad
T_n := \bigoplus\nolimits_{i \in \ZZ} K v_i,
\quad
  x^{\pm 1}\cdot v_i:= v_{i\pm n},
\endeq
with explicit $H_i$-actions.
Consider the $K[x^{\pm 1}]$-submodule $V_j := \bigoplus\nolimits_{i \in \ZZ} K v_{j+ni} \subseteq T_n$. 
Any $m\in V_j^{\otimes d}$ can be written as $P \cdot v_j^{\otimes d}$ for some polynomial $P \in K[x^{\pm1}]^{\otimes d}$.
The tensor product $V_j^{\otimes d}$ becomes (see \cite{LM25}) an $\widehat{\cH}_q(d)$-subdmodule of $T_{n}^{\otimes d}$ with
\eq\label{eq:KMS1}
\begin{split}
b \cdot (P\cdot v_j^{\otimes d}) 
	&:= (bP)\cdot v_j^{\otimes d},
\\
H_i \cdot (P\cdot v_j^{\otimes d}) 
	&:= q \sigma_i(P) \cdot v_j^{\otimes d}
	+ \rho_i(P)  \cdot v_j^{\otimes d}.
\end{split}
\endeq
It turns out that $V_j^{\otimes d}$ can be regarded as  a wreath module 
$V_j^{\otimes d} \equiv V_j \wr K_q$, and hence
the Kashiwara-Miwa-Stern action \eqref{eq:KMS1} can be translated  into the following,
for all $P\in K[x^{\pm1}]^{\otimes d}$ and  $y \in N$:
\eq\label{eq:KMS2}
H_i \cdot ((P \cdot v_j^{\otimes d})\otimes y) 
= (\sigma_i(P) \cdot v_j^{\otimes d}) \otimes (h_i \cdot y) 
+ (\rho_i(P)  \cdot v_j^{\otimes d}) \otimes y.
\endeq
\endexa
We immediately see from  \eqref{eq:KMS1} that one cannot expect \eqref{def:naiveaction} to hold as in the case \eqref{eq:HuAction}  for the Hu algebras. 
%
\section{The \Based Modules}
In this section the focus is on the conditions on the $B$-module $M$ such that the $K$-module
$
M\wr N := M^{\otimes d} \otimes N
$
affords a $B \wr \cH(d)$-module structure.
Such $B$-modules are called the {\em \based modules}.
\subsection{Conventions}
When it is unambiguous, the following abbreviations will be employed:
\eq
\cM := M^{\otimes d}, 
\quad
W := K\Sigma_d,
\quad
\cB := B^{\otimes d},
\quad
\tcH := \tcH^\chi_d,
\endeq
where $\tcH^\chi_d \cong E \wr \cH(d)$ is the twisted Hecke algebra with respect to a fixed twist $\chi = (S^M, R^M)$, under the assumption \eqref{asm:twistchi}.

Since we assumed that the algebra $B \wr \cH(d)$ has PBW bases, one can define an  invertible map 
\eq
\tau^B:  W \otimes \cB \to  \cB \otimes W,
\quad
w \otimes b \mapsto \sum\nolimits_g b_g \otimes g, 
\endeq
where the PBW monimial $H_w b$ is expressed as $\sum_g  b_g H_g$ via the other PBW basis.
Next, the multiplication of $B \wr \cH(d)$ leads to a map
\eq
m : \cB \otimes W \otimes \cB \otimes W \to \cB \otimes W,
\quad
b_1 \otimes w_1 \otimes b_2 \otimes w_2  \mapsto \sum\nolimits_g b_g \otimes g
\endeq
where $b_1 H_{w_1} b_2 H_{w_2} = \sum_g b_g H_g \in B \wr \cH(d)$.
One defines the following multiplication maps in a similar fashion:
\eq
m_W: W\otimes W \to \cB \otimes W, 
\quad
m_B:\cB\otimes \cB \to \cB,
\quad
\textup{ and }
\quad
m_\cH: \tcH \otimes \tcH \to \tcH.
\endeq 
Let  $\alpha_M: \cB \otimes \cM \to \cM$ (resp. $\alpha_N: \tcH \otimes N \to N$) be the structure map of $M$ as a $\cB$-module (resp. of $N$ as an $\tcH$-module).

If $M$ is equipped with a {\em structure map} $\tau^M : W\otimes \cM \to \cM \otimes \tcH$ (to be defined shortly), then
one can show that \eqref{def:QWPAction} gives a $B\wr \cH(d)$-action on $M\wr N$.

\subsection{Construction of the Structure Map $\tau^M$}
Similar to the approach in \cite[\S5]{LNX24}, we need to first fix an expression $r(w)$ for all $w\in \Sigma_d$ that is compatible with others, i.e., $r(1) = 1$ and for every non-identity element $w \in \Sigma_d$, there exists a unique $x\in \Sigma_d$ such that 
\eq\label{def:choicemap}
w = s_i x > x,
\quad
r(w) = s_i r(x)
\quad
\textup{for some }
i.
\endeq

\begin{Def}
Suppose that $\sigma^M$ and $\rho^M$ are two given maps in $\End_K(M \otimes M)$.
The structure map $\tau^M = \tau^M(r)$, which also depends on $(\sigma^M, \rho^M)$, is defined recursively by 
\eq\label{def:tauaction}
\tau^M(w\otimes m)
	:= \begin{cases}
	m \otimes 1, & \tif w = 1;
        \\
        \sigma^M_i(m)\otimes h_i + \rho_i^M(m)\otimes 1, &\tif w = s_i;
        \\
        (\id\otimes\mu)\circ(\tau^M\otimes\id)(s_i\otimes \tau^M(s_iw \otimes m)),
        & \tif  r(w) = s_ir(x). \\
	\end{cases}
\endeq
\end{Def}
Equivalently, $\tau^M(r)$ is defined inductively with respect to the filtration $(W^{(\ell)}\otimes \cM)_{\ell \geq 0}$ on $W\otimes \cM$, where $W^{(\ell)} := \bigoplus_{\ell(w)=\ell} K w$.
To be precise, 
define $\iota_r : W^{(\ell+1)}\to W^{(1)} \otimes W^{(\ell)}$ by $\iota_r(w ) = s_i \otimes x$ for the unique $s_i$ satisfying $r(w) = s_i r(x)$.
We begin with defining  $\tau^M_{\ell}$ on $W^{(\ell)}\otimes \cM$ via \eqref{def:tauaction} if $\ell  \leq 1$.
Then, the third line of \eqref{def:tauaction} translates into that $\tau^M_{\ell+1} $ is the following composition:
\eq\label{def:tauMdiag}
\xymatrix{
W^{(\ell+1)}\otimes \cM \ar[d]^{\iota_r \otimes \id}
&&
&&
\cM\otimes \tcH
\\
W^{(1)}\otimes W^{(\ell)}\otimes \cM \ar[rr]^{\id \otimes \tau^M_{\ell}}
&&
W^{(1)}\otimes \cM \otimes \tcH  \ar[rr]^{\tau^M_{1} \otimes \id}
&&
\cM \otimes \tcH \otimes \tcH  \ar[u]^{\id\otimes m_\cH}
}
\endeq
Recall that  when $d = 2$, $r$ is unique and hence  $\tau^M$  does not depend on a choice of $r$. 
\exa
Consider the longest element $w_\circ = s_1s_2s_1 = s_2 s_1 s_2 \in \Sigma_3$.
If we choose $r(w_\circ) = s_2 s_1 s_2$, then $\tau^M(w_\circ\otimes m)$ has to be defined based on the value of $\tau^M(s_1s_2\otimes m) = \sum_w m_w \otimes h_w$, i.e.,
\eq
\tau^M(r)(w_\circ) := \sum\nolimits_w (1\otimes \mu)(\tau^M(s_2 \otimes m_w) \otimes h_w).
\endeq
However, for a different choice $r'$ with $r'(w_\circ) = s_1 s_2 s_1$, one gets
\eq
\tau^M(r')(w_\circ) := \sum\nolimits_w (1\otimes \mu)(\tau^M(s_1 \otimes m'_w) \otimes h_w),
\quad\textup{where}
\quad
\tau^M(s_2s_1\otimes m) = \sum\nolimits_w m'_w \otimes h_w.
\endeq
Thus, $\tau^M(r)$ may not be independent of choice of $r$ in general. 
\endexa
\subsection{\Based Modules over the Base Algebra $B$}
Recall that when we refer a quantum wreath product $B \wr \cH(d)$,
we actually mean a tuple $({B}, \sigma, \rho, S, R)$ and an integer $d\geq 1$.
Now, we introduce an analogous notion for the modules over $B \wr \cH(d)$.
Consider the tuple $(M, \sigma^M, \rho^M, S^M, R^M)$ consisting of
\begin{itemize}
\item A left $B$-module $M$,
\item Two maps $\sigma^M, \rho^M \in \End_K(M \otimes M)$ that are $K$-linear, and 
\item Two maps $S^M, R^M \in E^{\otimes 2} \subseteq \End_{B\otimes B}(M \otimes M)$ that are $B\otimes B$-linear.
\end{itemize}
This tuple is equivalent to the triple $(M, \tau^M, \tcH^\chi_d)$, where
\begin{itemize}
\item $\tcH^\chi_d = E \wr \cH(d)$ is the twisted Hecke algebra satisfying \eqref{asm:twistchi};
\item $\tau^M: W\otimes \cM \to \cM\otimes \tcH^\chi_d$ is the structure map \eqref{def:tauaction} with respect to a choice $r$.
\end{itemize}
Recall the following maps defined earlier: $\tau^B: W\otimes \cB\to \cB \otimes W$, 
$m_W: W\otimes W \to \cB \otimes W$, 
and $\alpha_M : \cB\otimes \cM \to \cM$.
Consider the following diagrams that depend on the structure map 
$\tau^M$:
\eq\label{fig:basedmod}
\begin{array}{cc}
\xymatrix{
W \otimes \cB\otimes \cM \ar[rr]^{\tau^B\otimes \id} \ar[d]^{\id\otimes \alpha_M} 
&&\cB\otimes W\otimes \cM 
\ar[d]^{\id\otimes {\tau}^M}
\\
W \otimes \cM \ar[d]^{\tau^M}
&& \cB\otimes \cM\otimes\tcH \ar[lld]^{\alpha_M \otimes \id}
\\
\cM\otimes\tcH & 
}
&
\xymatrix{
W\otimes W\otimes \cM \ar[d]^{\id\otimes\tau^M} \ar[rr]^{m_W\otimes \id} 
&&\cB \otimes W\otimes \cM \ar[d]^{\id \otimes \tau^M}
\\
W\otimes \cM\otimes\tcH \ar[d]^{\tau^M\otimes \id}
&& \cB \otimes \cM\otimes \tcH \ar[d]^{\alpha_M \otimes \id}
\\
\cM\otimes \tcH\otimes\tcH \ar[rr]^{1\otimes m_\cH} 
 &&
\cM\otimes\tcH 
}
\\
\normallabel{WB}  \mbox{(WB)}
	&\normallabel{WW}  \mbox{(WW)}
\end{array}
\endeq
We will refer to the diagram to the left (and right, resp.) in \eqref{fig:basedmod} as \hyperref[WB]{(WB)} (and \hyperref[WW]{(WW)}, resp.).

\Def\label{def:basedmod}
A left $B$-module $M$ (or, a tuple $(M, \sigma^M, \rho^M, S^M, R^M)$) is called a {\em \based module} if 
\enu[(i)]
\item
The map $\tau^M(r)$ is independent of the choice of $r$;
\item
Both diagrams in \eqref{fig:basedmod} commute.
\endenu


In order to prove a criteria for $M$ to be a \based module, let us begin with certain conditions regarding the maps $\sigma^M$ and $\rho^M$ which are counterparts of the conditions \eqref{def:P1}--\eqref{def:P9} for the PBW basis theorem.
Let $l^M_S, l^M_R \in \End_K(M \otimes M)$ be the left action by $S, R \in B$ on $M \otimes M$.
Consider the equations, for
$1\leq i \leq d-1$,
\  $m\in M^{\otimes d}, 
\ b\in B^{\otimes d}$,
\ $|i-j| =1$, 
and $|k - l|>1$:
\begin{align}
\tag{M2}&\label{def:M2} \sigma^M_i(b\cdot m) = \sigma_i(b)\cdot \sigma^M_i(m),
\quad
\rho_i^M(b\cdot m) =  \sigma_i(b) \cdot \rho^M_i(m) + \rho_i(b) \cdot m,
\\
\tag{M4}&\label{def:M4}
\SK (\sigma^M_i)^2+\rho^M_i\sigma^M_i+\sigma^M_i\rho^M_i = l^M_{S_i}\sigma^M_i,
\quad
\RK (\sigma^M_i)^2+(\rho^M_i)^2 = l^M_{S_i} \rho^M_i + l^M_{R_i}.
\end{align}
The next three conditions will be considered if $d \geq 3$:
\begin{align}
\tag{M5}&\label{def:taubr1}\sigma^M_i\sigma^M_j\sigma^M_i = \sigma^M_j \sigma^M_i \sigma^M_j,
\quad \rho^M_i\sigma^M_j\sigma^M_i = \sigma^M_j \sigma^M_i \rho^M_j,
\\
\tag{M6}&\label{def:taubr2}
\rho^M_i\sigma^M_j\rho^M_i =\SK \sigma^M_j \rho^M_i \sigma_j^M + \rho_j^M\rho_i^M\sigma_j^M + \sigma_j^M \rho_i^M \rho_j^M,
\\
\tag{M7}&\label{def:taubr3}\rho_i^M\rho_j^M\rho_i^M + \RK \sigma_i^M \rho_j^M \sigma_i^M
= \rho_j^M\rho_i^M\rho_j^M + \RK \sigma_j^M \rho_i^M \sigma_j^M.
\end{align}

\begin{thm}\label{thm:basedmod}
$(M, \sigma^M, \rho^M, r^M_S, r^M_R)$ is a \based $B$-module if and only if the following are true:
\begin{enumerate}[\textup{(}i\textup{)}]
\item \eqref{def:M2} and \eqref{def:M4} hold;
\item \eqref{def:taubr1}--\eqref{def:taubr3} hold additionally, if $d\geq 3$.
\end{enumerate}
\end{thm}
\proof
For the necessary part, we assume that $B$ is a \based module.
Then, both (i) and (ii) follow from a direct computation (see Section \ref{app:neci} for (i) and Section \ref{app:necii} for (ii)).

For the sufficient part:
\begin{itemize}
\item Step 1: We show that $\tau^M(r)$ is independent of choice of $r$ if  \eqref{def:taubr1}--\eqref{def:taubr3} hold. 
This claim follows from another direct computation which can be found in Section \ref{app:indep}.
\end{itemize}
For the next steps, consider the following diagrams for $k, \ell \geq 1$:
\eq\label{fig:basedmodl}
\xymatrixrowsep{1.5pc}
\xymatrixcolsep{1pc}
\begin{array}{cc}
\xymatrix{
W^{(\ell)} \otimes  \cB \otimes \cM \ar@{=}[r] \ar[d]_{\id \otimes \alpha_M}
	&W^{(\ell)} \otimes  \cB \otimes \cM \ar[d]^{\tau^B\otimes \id}
\\
W^{(\ell)} \otimes  \cM \ar[d]_{\tau^M}
	&\cB \otimes W^{(\ell)} \otimes \cM  \ar[d]^{\id\otimes \tau^M} 
\\
\cM \otimes\tcH \ar@{=}[rd]& \cB \otimes \cM \otimes \tcH \ar[d]^{\alpha_M\otimes \id}
\\
	&\cM\otimes \tcH
}
&
\xymatrix{
W^{(k)} \otimes W^{(\ell)} \otimes \cM \ar[d]_{\id\otimes \tau^M} \ar@{=}[r]
	&W^{(k)} \otimes W^{(\ell)}\otimes  \cM \ar[d]^{m_W\otimes \id}
\\
W^{(k)} \otimes \cM \otimes \tcH \ar[d]_{\tau^M\otimes \id}
	&\cB \otimes  W^{(k)} \otimes \cM \ar[d]^{\id\otimes\tau^M}
\\
\cM \otimes \tcH \otimes \tcH \ar[dr]_{\id\otimes m_\cH}
	& \cB \otimes  \cM \otimes  \tcH \ar[d]^{\alpha_M \otimes \id}
\\
	&\cM \otimes \tcH 
}
\\
\normallabel{WBl}  \mbox{(WB$^\ell$)}
& \normallabel{WWkl}  \mbox{(WW$_k^\ell$)}
\end{array}
\endeq
We will refer to the diagram to the left (and right, resp.) in \eqref{fig:basedmodl} as 
\hyperref[WBl]{(WB$^\ell$)}  (and \hyperref[WWkl]{(WW$_k^\ell$)}, resp.).

\begin{itemize}
\item Step 2: We show that all \hyperref[WBl]{(WB$^\ell$)} commute using  an induction on $\ell$. 
It follows once again from the calculation in Section \ref{app:neci} that  \eqref{def:M2} implies the commutativity of  \hyperref[WBl]{(WB$^1$)}.
The induction proceeds since
\[
 \textup{Commutativity of \hyperref[WBl]{(WB$^1$)}, \hyperref[WBl]{(WB$^\ell$)} }
\Rightarrow
\textup{commutativity of \hyperref[WBl]{(WB$^{\ell+1}$)}},
\]
whose verification can be found in Section \ref{app:WBcom}.
\item Step 3: To prove that \hyperref[WW]{(WW)} commutes, 
consider first the case when the two composite maps are evaluated at $w\otimes x \otimes m$  for any $m \in \cM$ and $x, w \in \Sigma_d$ such that $\ell(wx) = \ell(w) + \ell(x)$.
That is, the multiplication map $m_W$ takes $w\otimes x$ to $1\otimes (wx) \in \cB \otimes W^{\ell(wx)}$, and hence
\eq
\begin{split}
&(\alpha_M\otimes \id_\cH \circ \id_\cB \otimes \tau^M \circ m_W \otimes \id_\cM) (w\otimes x \otimes m) = \tau^M(wx \otimes m) 
\\
&\qquad= (\id_\cM \otimes m_\cH\circ \tau^M \otimes \id_\cH \circ \id_W \otimes \tau^M)(wx),
\end{split}
\endeq
where the last equality follows from \eqref{def:tauMdiag} together with the independence of choice of $r$ proved in Step 1.
It remains to consider the case where $\ell(wx) < \ell(w) + \ell(x)$, i.e., we may assume for now $m_W(W^{(1)}\otimes W^{(\ell)}) \subseteq \cB \otimes W^{\ell}$.

We use the commutativity of \hyperref[WBl]{(WB$^k$)}  to start a double induction on $(k, \ell)$ to show that all \hyperref[WWkl]{(WW$_k^\ell$)} commute.
For the base cases, \eqref{def:M4} translates into that  \hyperref[WWkl]{(WW$_1^1$)} commutes.
Additionally, it is immediate that \hyperref[WWkl]{(WW$_k^0$)} and \hyperref[WWkl]{(WW$_0^\ell$)}  commute for all $k \geq 0$.
Then, the induction proceeds since
\eq\label{eq:doubleinduction}
\textup{Commutativity of 
\hyperref[WBl]{(WB$^k$)},
\hyperref[WWkl]{(WW$_k^\ell$)},
\hyperref[WWkl]{(WW$_1^\ell$)} }
\Rightarrow 
\textup{
Commutativity of 
\hyperref[WWkl]{(WW$_k^{\ell+1}$),} }
\endeq
whose verification can be found in Section \ref{app:WWcom}.
\end{itemize}
Therefore, both diagrams \eqref{fig:basedmod} commute, and hence $M$ is a \based module.
\endproof

%

\subsection{Examples of \Based Modules}
In this section, we consider several important algebras in the literature that are quantum wreath products.
Using Theorem \ref{thm:basedmod}, we will demonstrate a vast family of modules that are wreath modules constructed from certain \based modules.
\subsubsection{(Generalized) Hu Algebras}\label{sec:gHu}
Let $M = S_m^\lambda$ be the Specht module over $B = \cH_q(\Sigma_m)$ with respect to the partition $\lambda$.
Let $\chi  = \chi_{\lambda}$ from Example \ref{ex:HuClifford}. 
By comparing \eqref{eq:HutauM} with \eqref{def:tauaction}, we see that for the modules $S(\lambda, \lambda)_{\pm} = M \wr K_{\pm \sqrt{f_{\lambda}}}$, the structure map is an isomorphism of $K$-modules given by, for $w\in \Sigma_2, z\in M\otimes M$:
\eq\label{eq:HutauM2}
\tau^M(w \otimes z) = z \otimes h_w,
\quad
(\textup{i.e., }
\sigma^M = \textup{flip},
\ 
\rho^M = 0).
\endeq
Indeed, $S^\lambda_m$ is a \based module with respect to $\tau^M$ given in \eqref{eq:HutauM2}, which follows from 
\[
\sigma^M((T_x\otimes T_y) \cdot (m_1\otimes m_2)) 
=   (T_y \otimes m_2)\otimes (T_x\cdot m_1)
= \sigma(T_x\otimes T_y) \cdot \sigma^M(m_1\otimes m_2),
\] 
for all $x, y \in \Sigma_m$, $m_i \in M$,
and the fact that the rest of equalities in \eqref{def:M2}, \eqref{def:M4} hold on the spot.

Furthermore, if we consider the generalized Hu algebra $B \wr \cH(d)$, $d\geq 2$ with the same choice of parameters as in the case of the Hu algebra.
We see that \eqref{def:taubr1}--\eqref{def:taubr3} all hold immediately, and hence $\tau^M$ for $M =S^\lambda_m$ is independent of the choice of $r$.

\subsubsection{Affine Hecke Algebras}\label{sec:AHA}
Following Example \ref{ex:AHAClifford}, $\tcH^\chi_d$ is the usual Hecke algebra $\cH_q(\Sigma_d)$ with generators $h_1, \dots h_{d-1}$ subject to the Type A braid relations as well as the quadratic relation $h_i^2 - \SK h_i - \RK = 0$, where $
\SK = q-1, \RK = q \in K$.

For the module $V_j^{\otimes d} = M \wr K_q$ with $M = V_j$, the structure map $\tau^M$ in \eqref{eq:KMStauM} translates into that, for all $i$, $P \in K[x^{\pm 1}]^{\otimes d}$:
\eq\label{eq:AHAtauM2}
\sigma_i^M(P\cdot v_j^{\otimes d}) = \sigma_i(P) \cdot v_j^{\otimes d},
\quad 
\rho_i^M(P\cdot v_j^{\otimes d}) = \rho_i(P)\cdot v_j^{\otimes d}.
\endeq
With our setup, one can give an alternative proof that the Kashiwara-Miwa-Stern action is well-defined by checking \eqref{def:taubr1}--\eqref{def:taubr3} and \eqref{def:M2}--\eqref{def:M4}.
Indeed, in this case, 
\eqref{def:taubr1}--\eqref{def:taubr3}, \eqref{def:M2} and \eqref{def:M4} follow from \eqref{def:P5}--\eqref{def:P7}, \eqref{def:P2} and \eqref{def:P4}, respectively.

\subsubsection{Ariki--Koike Algebras}\label{sec:mAK}
let $u_1, \dots, u_m$ be elements in $K$ such that $\Delta:= \prod_{i>j} (u_i - u_j)$ is invertible.
Let $F_1, \dots, F_m$ be the unique polynomials in $K[X]$ of degree $m-1$ determined by $F_i(u_j) = \delta_{i,j} \Delta$ for all $1\leq i, j \leq m$, i.e.,
the coefficient matrix $(f_{i,j})_{1\leq i,j \leq m}$ obtained from $F_i(X) = \sum_{j=1}^m f_{i,j} X^{j-1}$ is obtained from 
\eq
(f_{i,j})_{1\leq i,j \leq m} = \Delta 
\left(\begin{smallmatrix}
1 & 1& \dots & 1 
\\
u_1 & u_2 & \dots & u_m
\\
\vdots & \vdots & \vdots & \vdots
\\
u_1^{m-1} & u_2^{m-1} & \dots & u_m^{m-1}
\end{smallmatrix}\right)^{-1}.
\endeq
Consider the (modified) Ariki--Koike algebra $B\wr\cH(d)$ with $B = \frac{K[t]}{\prod_{i=1}^m(t-u_i)}$, 
$S = v-v\inv$,
$R = 1$,
and $\rho$ determined by
\[
\rho(t\otimes 1) = -\frac{S}{\Delta^2} \sum\nolimits_{i<j} (u_j - u_i) F_i(t\otimes 1) F_j(1\otimes t).
\]
Let $M = K_a$ be the one-dimensional simple module over $B$ on which $t$ acts by the multiplication by $u_a$ for some $1\leq a \leq m$. 
Next, since $S, R \in K(1\otimes 1)$, $\cH_\chi^d$ is again the usual Hecke algebra $\cH_q(\Sigma_d)$ with generators $h_1, \dots h_{d-1}$ subject to the Type A braid relations as well as the quadratic relation $h_i^2 - \SK h_i - \RK = 0$, where $\SK = v-v^{-1}, \RK = 1 \in K$.

{In the following, we determine all structure maps $\tau^M$ that satisfy the assumption} $\sigma^M_i(k) = s k$ and $\rho^M_i(k) = r k$ for all $k \in M^{\otimes d}$, for some $s, r \in K$.
Next, \eqref{def:taubr2} becomes
$0 =sr(s + r) $;
while 
\eqref{def:taubr3} is automatically true.

Applying 
$b = t_1$ to \eqref{def:M2} yields a vacuous statement and that $0 = \rho(t_1 \cdot k)$, or,
\eq
0  = -\frac{S}{\Delta^2} \sum\nolimits_{i < j} F_i(u_a) F_j(u_a),
\endeq
which holds since $ F_i(u_a)$ and $ F_i(u_a)$ cannot be nonzero simultaneously.
Finally, \eqref{def:M4} becomes 
\eq
s( (v-v^{-1})(s-1)+2r) = 0,
\quad
s^2 + r^2 = (v-v^{-1}) r + 1.
\endeq
An elementary calculation shows that $\tau^M$ is a structure map if $s=0$ and $r = \gamma$ with $\gamma \in \{ v, -v^{-1}\}$.

\section{The Wreath Modules}
\subsection{The Construction of the Wreath Modules}\label{sec:WM}
We start with a fundamental result that shows that the wreath module has a well-defined action. The main idea is to check the commutativity of the diagrams to 
prove that the action is associative. 

\thm\label{thm:MwrN}
Suppose that  $M$ is a \based left $B$-module, and $N$ is a left module over $\tcH = \tcH^{\chi}_d$.
Then, the map $\alpha_{MN}$ in \eqref{def:QWPAction} gives a $B\wr \cH(d)$-action on $M\wr N$.


\proof

We will use commutative diagrams in \eqref{fig:basedmod} to show that the following diagram  commutes:
\eq\label{fig:comm0}
\xymatrix{
\cB  W\cB  W \cM N \ar[rr]^{\id \id \alpha_{MN}} \ar[d]^{m\id\id}
&&\cB  W \cM N  \ar[d]^{\alpha_{MN}}
\\
\cB  W \cM N  \ar[rr]^{\alpha_{MN}} 
&&
\cM N
}
\endeq
Consider the following diagram:
\eq\label{fig:bigcomm}
\xymatrixrowsep{1pc}
\xymatrixcolsep{1pc}
\xymatrix{
\cB  W\ud{\alpha_{MN}}{\cB  W \cM N} \ar[dd] \ar@{=}[r]
	&\cB  W\cB  \ud{\tau^M}{W \cM} N  \ar[d] \ar@{=}[r]
		&\cB  \ud{\tau^B}{W\cB}  \ud{\tau^M}{W \cM}N \ar[d] \ar@{=}[r]
			&\ud{m_B\circ \id\tau^B}{\cB  W\cB}  W \cM N \ar[d]\ar@{=}[r]
				&\ud{m_B m_W \circ \id \tau^B \id}{\cB W \cB W} \cM N \ar[d]
\\
	&\cB  \ud{\substack{\tau^M\circ \id\alpha_M\\ = \alpha_M\id\circ\id\tau^M\circ\tau^B\id}}{W \cB  \cM} \tcH N  \ar[d]
		&\cB \cB \ud{\tau^M}{W \cM} \tcH N \ar[d]
			&\cB \ud{\substack{\id m_\cH \circ \tau^M \id\circ \id \tau^M\\ =\alpha_M\id \circ \id\tau^M \circ m_W\id}}{W W M} N \ar[dd]
				&\cB \cB \ud{\tau^M}{W \cM} N \ar[d]
\\
\ud{\alpha_{MN}}{\cB  W \cM N} \ar@/_2.5pc/[ddrr]
	& \cB \cM \cH \ud{\alpha_N}{\tcH N}\ar[d]
		&\ud{\substack{\alpha_M\circ\id\alpha_M \\ =\alpha_M \circ m_B\id}}{\cB \cB\cM}  \udr{\substack{\alpha_N\circ\id\alpha_N \\ =\alpha_N \circ m_\cH\id}}{\tcH \tcH N} \ar[dd]
			&
				& \ud{\substack{\alpha_M\circ\id\alpha_M \\ =\alpha_M \circ m_B\id}}{\cB \cB\cM}  \ud{\alpha_N}{\tcH N} \ar@/^2.5pc/[ddll]
\\
	&\ud{\alpha_{M}}{\cB  \cM}\ud{\alpha_N}{ \tcH N } \ar[dr]
		&
			&\ud{\alpha_{M}}{\cB  \cM}\ud{\alpha_N}{ \tcH N }\ar[dl]
				&
\\
&&\cM N&&
}
\endeq
Thanks to the commutativity of  \eqref{fig:basedmod}, one has the identifications of maps
$\tau^M\circ \id\alpha_M = \alpha_M\id \circ \id\tau^M\circ \tau^B\id : W\cB \cM \to \cM \tcH$ 
and 
$\id m_\cH \circ \tau^M \id\circ \id \tau^M = \alpha_M\id \circ \id\tau^M \circ m_W\id : WW \cM \to \cM \tcH$.
Also, since $M$ is a $B$-module, $\alpha_M \circ \id \alpha_M = \alpha_M\circ m_B\id : \cB \cB \cM \to \cM$.
Similarly, $\alpha_N \circ \id \alpha_N = \alpha_N\circ m_\cH\id : \tcH \tcH N \to N$ since $N$ is an $\tcH$-module.
That is, the diagram is well-defined.

Note that there are five paths from $\cB W \cB W \cM N$ to $\cM N$.
To show the the diagram \eqref{fig:bigcomm} commutes, we need to check that the composite map produced from any such path coincides with the composite map produced from the path to its right. 
Here, we omit the verification (see \ref{app:bigcomm} for details).

Finally, note that $m = m_B\id \circ m_B m_W\circ \id \tau^M \id: \cB W \cB W \to \cB W$,  the map produced from the rightmost path in \eqref{fig:bigcomm} is exactly the map $\alpha_{MN}\circ m\id\id $,
and thus we conclude that \eqref{fig:comm0} commutes.

\endproof

\subsection{Induced Modules as Wreath Modules}
Recall \eqref{eq:ind}.  
Consider the induced module 
\eq
\ind_{1^d}^d(M^{\otimes d}) =  (B\wr \cH(d))  \otimes_{B^{\otimes d}} M^{\otimes d},
\endeq
on which the $B \wr \cH(d)$-action is given by, for $b\in \cB, \ w,x\in \Sigma_d, \ m\in \cM$:
\eq
(bH_w) \cdot (H_x \otimes m) 
:=
\sum\nolimits_z H_z \otimes (b_z  \cdot m),
\quad
\textup{if}
\quad
bH_w H_x = \sum\nolimits_z H_z b_z.
\endeq
In other words, the structure map $\cB \otimes W \otimes W \otimes \cM \to W \otimes \cM$ of $\ind_{1^d}^d(\cM)$ is given by
\eq
\xymatrix{
\cB\otimes W\otimes W\otimes \cM \ar[d]^{\id \otimes m_W\otimes \id}
	&&
		&&W \otimes \cM
\\
\cB \otimes \cB \otimes  W \otimes \cM \ar[rr]^{\id \otimes (\tau^B)\inv \otimes  \id}
	&&\cB \otimes W \otimes \cB\otimes  \cM \ar[rr]^{(\tau^B)\inv\otimes \alpha_M}
		&& W \otimes \cB\otimes  \cM \ar[u]^{\id \otimes\alpha_M}
}
\endeq
%

%

\lem\label{lem:tauMinv}
Suppose that $M$ is a \based module.
If $\sigma^M_i$ is invertible for all $i$, then $\tau^M$ is invertible. 
\endlem

\proof
For any $m\in \cM, w\in \Sigma_d$, expressing the element $\tau^M(w \otimes m)$ in terms of the  left $\cB$-basis $\{1 \otimes h_w ~|~ w\in \Sigma_d\}$, we get $\tau^M(w \otimes m) = \sum_{x \leq w} \upsilon_{w, x}(m) \otimes h_x$ for some $\upsilon_{w,x} \in \cB$.
Let $w = s_{i_1} \dots s_{i_N}$ be a reduced expression, the map $\sigma^M_w := \sigma^M_{i_1} \dots \sigma^M_{i_N}$ is well-defined by (M5).
Note that by construction of $\tau^M$, $\upsilon_{w,w} = \sigma^M_w$, and is invertible by assumption.

Since the domain $K \Sigma_d \otimes M^{\otimes d}$ has a right $\cB$-basis $\{w \otimes 1 ~|~ w\in \Sigma_d\}$, 
the matrix form of the map $\tau^M$ with respect to these two $\cB$-bases is a triangular matrix over $\End_K(\cB)$, with $\upsilon_{w, w} = \sigma^M_w$ on the diagonal. The invertibility of $\tau^M$ follows from that of $\sigma^M_{w}$.

\endproof

The following key result allows one to study the stucture of the induced modules via our wreath module construction. 

\begin{theorem}\label{thm:indhom}
Suppose that $M$ is a \based $B$-module. 
Then, the structure map $\tau^M: W \otimes M^{\otimes d} \to M^{\otimes d} \otimes \tcH_d^\chi$, $w \otimes m \mapsto \sum_x \sigma_{x,w}^M(m) \otimes h_x$ defines a $B \wr H(d)$-module homomorphism also denoted by $\tau^M$:
\[
\tau^M : \ind_{1^d}^d (M^{\otimes d}) \to M \wr \tcH^\chi_d,
\quad
 H_w \otimes m \mapsto \sum\nolimits_x \sigma_{x,w}^M(m) \otimes h_x.
\]
In particular,  $\tau^M$ is an isomorphism if $\sigma^M_i$ is invertible for all $i$.
\end{theorem}
\proof
The proof of that $\tau^M$ is a $B\wr \cH(d)$-module homomorphism can be done via a direct verification using the commutativity of both \hyperref[WB]{(WB)} and \hyperref[WW]{(WW)}. See Section \ref{app:indhom} for details. 
The statement about the isomorphism follows from the invertibility of $\tau^M$ in Lemma \ref{lem:tauMinv}.
\endproof

\subsection{Wreath of Morphisms}

Our goal is to investigate the $B\wr\cH(d)$-module homomorphisms between two  wreath modules $M_1 \wr N_1$ and $M_2 \wr N_2$,
where $N_1$ and $N_2$ are modules over the same twisted Hecke algebra $\tcH = \tcH^\chi_d$. 
Define the following $K$-modules:
\eq
\begin{split}
&\Hom_{B}(M_1, M_2) \wr
\Hom_{\tcH}(N_1, N_2)
\\
&\quad:= 
\textup{Span}\{ f_1 \otimes \dots \otimes f_d \otimes g \in \Hom_K(M_1\wr N_1, M_2 \wr N_2) 
  ~|~
f_i \in \Hom_B(M_1, M_2),\ 
g \in \Hom_{\tcH}(N_1,N_2)
\}.
\end{split}
\endeq
It turns out that one does need restrictive conditions to ensure that $\Hom_{B}(M_1, M_2) \wr
\Hom_{\tcH}(N_1, N_2)
=
\Hom_{B\wr \cH(d)}(M_1\wr N_1, M_2 \wr N_2)$.

\begin{prop}\label{prop:Hom}
Given $m \in M^{\otimes d}_j$ and $w, x\in \Sigma_d$, let $m^w_x$ be the element in $M_j^{\otimes d}$ such that 
\[
\tau^{M_j}(w\otimes m) = \sum\nolimits_{x\in \Sigma_d} m^w_x \otimes h_x.
\]
Write $f = f_1\otimes \dots \otimes f_d$ where $f_k \in \Hom_{K}(M_1, M_2)$ for $1\leq k \leq d$.
\begin{enua}
\item
Suppose that 
$g\in \Hom_{\tcH}(N_1,  N_2)$,
$f_k \in \Hom_{B}(M_1, M_2)$ for $1\leq k \leq d$.
If  $f_k(m^w_x) = f_k(m)^w_x$ for all $k, m, w, x$, then
$f \otimes g \in \Hom_{B\wr \cH(d)}(M_1\wr N_1, M_2 \wr N_2)$. That is,
\eq\label{eq:dimcompleft}
\Hom_{B}(M_1, M_2) \wr
\Hom_{\tcH}(N_1, N_2)
\subseteq
\Hom_{B\wr \cH(d)}(M_1\wr N_1, M_2 \wr N_2).
\endeq
In particular, \eqref{eq:dimcompleft} holds if $\tau^{M_j}$'s are of the form $\tau^{M_j}(w\otimes m) = \sigma^{M_j}_w(m)\otimes h_w$ where the maps $\sigma^{M_j}_w$ are all invertible. 
\item 
If $\dim \Hom_B(M_1, M_2) \leq 1$,
then 
\[
\Hom_{B}(M_1, M_2) \wr
\Hom_{\tcH}(N_1, N_2)
\cong
\Hom_{B\wr \cH(d)}(M_1\wr N_1, M_2 \wr N_2).
\]
\end{enua}
\end{prop}
\proof
For part (a), the assumptions imply the diagram below commutes:
\eq
\xymatrix{
\cB \otimes W\otimes\cM_1 \otimes N_1 
	\ar[d]^{\id  \otimes \id \otimes f \otimes g} 
	\ar[rr]^{\id \otimes \tau^{M_1}\otimes \id} 
&& \cB \otimes\cM_1\otimes\tcH\otimes N_1
	\ar[d]^{\id  \otimes f \otimes \id \otimes g} 
	\ar[rr]^{\alpha_{M_1} \otimes \alpha_{N_1}}
&& \cM_1\otimes N_1
	\ar[d]^{f \otimes g} 
\\
\cB\otimes W\otimes\cM_2\otimes N_2 
	\ar[rr]^{\id \otimes \tau^{M_2}\otimes \id} 
&& \cB\otimes \cM_2\otimes\tcH\otimes N_2
	\ar[rr]^{\alpha_{M_2} \otimes \alpha_{N_2}}
&& \cM_2\otimes N_2
}
\endeq
Hence, $f_1 \otimes \dots \otimes f_d \otimes g$ is a $B \wr\cH(d)$-module homomorphism. 

For part (b), if $\Hom_{B}(M_1, M_2) = 0$, then both sides are zero.
Assume now $\Hom_{B}(M_1, M_2)\cong K$. 
Then, for any $m\in M_1$, there is $m' \in M_2$ such that any $B$-module homomorphism $f$ takes $m$ to $cm'$ for some $c\in K^{\times}$.
That is, any $\phi \in \Hom_{B \wr \cH(d)}(M_1 \wr N_1, M_2 \wr N_2)$ is of the form $\phi(m\otimes n) = \sum_i m' \otimes g_i(n)$ for some $K$-linear homomorphism $g_i: N_1 \to N_2$. 

For each $w\in \Sigma_d$, recall that $\tau^{M_j}(w \otimes m) = \sum_x m^w_x \otimes h_x$. 
Since $\tau^{M_j}$ commutes with scalar multiplications, $(m^w_x)' = (m')^w_x$ for all $m \in M_1$ and $x,w \in \Sigma_d$.
By equating
$\phi(h_w \cdot (m\otimes n)) = \sum_x \phi(  m^w_x \otimes (h_x\cdot n)) =  \sum_{i,x}  (m^w_x)' \otimes g_i(h_x\cdot n) $ and
$h_w \cdot (\phi(m\otimes n)) = h_w \cdot \sum_i m' \otimes g_i(n) =  \sum_{i,x}  (m)'^w_x \otimes h_x\cdot g_i( n)$,
it follows that $g_i$'s are $\tcH$-module homomorphisms, which finishes the proof.
%
%
%
\endproof
%

\subsection{Wreath Functor}
Now, we show that the process of producing a wreath module with respect to the same \based $B$-module $M$ is functorial.

Fix a \based $B$-module $(M, \tau^M, \tcH^\chi_d)$, where we abbreviate the twisted Hecke algebra by $\tcH := \tcH^\chi_d$. 
Define the corresponding {\em wreath functor} $M\wr-: \cH\textup{-mod} \to B\wr\cH(d)\textup{-mod}$ by
\eq
 f \in \Hom_\cH(N_1,N_2) 
\quad
\mapsto
\quad
M\wr f:=  \id_\cM \otimes f \in \Hom_{B \wr \cH(d)}(M \wr N_1, M \wr N_2 ). 
\endeq
\prop\label{prop:exactness}
Suppose that $B$ is a \based module and $\End_B(M)\cong K$.
The wreath functor $M \wr - : \text{$\tcH$-mod} \to \text{$B \wr \cH(d)$-mod}$ is exact and fully faithful.
\endprop
\proof
Theorem \ref{thm:MwrN} asserts that for every $\cH$-module $N$, $M \wr N$ is a $B \wr \cH(d)$-module. 
Proposition \ref{prop:Hom} shows that for $f \in \Hom_{\tcH}(N_1, N_2)$, the corresponding morphism $M \wr f$ is indeed in $\Hom_{B \wr \cH(d)}(M \wr N_1, M \wr N_2)$. 
Hence, to verify that $M \wr -$ is a functor, it suffices to show that $M \wr (f \circ g) = (M \wr f) \circ (M \wr g)$ and $M \wr \id = \id$, which are done by a direct verification.

Let $F_0 : \cH\textup{-mod} \to K\textup{-mod}$ and $F_1 : B \wr \cH(d)\textup{-mod} \to K\textup{-mod}$ be the forgetful functors. 
Then, $(\cM \otimes - ) \circ F_0 = F_1 \circ (M \wr -)$, where $\cM \otimes -$ is an exact endofunctor on $K$-mod. 
Thus, the exactness of $M \wr -$ follows from the exactness of $\cM \otimes -$, due to the PBW basis theorem (see \eqref{asm:basis}).

Note that $M\wr-$ is fully faithful if and only if  $f \mapsto M\wr f$ defines a bijection between $\Hom_{\tcH}(N_1,N_2)$ and $\Hom_{B \wr \cH(d)}(M \wr N_1, M \wr N_2 )$,
which follows from Proposition \ref{prop:Hom}, since $\End_B(M)\cong K$.
\endproof

%
\subsection{Necessary Conditions for Non-Trivial Extensions}
We conclude this section by investigating the extensions of wreath modules. 
The following lemma establishes the necessary conditions for $\Ext^1$ between wreath modules to be nonzero. 
This property leads directly to a homological interpretation of the dominance order that will be defined in Definition \ref{def:domPO}. 

\begin{lem}\label{lem:gHuExt2}
Suppose that $S = 0$.
If $\Ext_{B\wr \cH(d)}^1(M_1 \wr N_1, M_2 \wr N_2) \neq 0$, where $M_i$'s are finite-dimensional \based $B$-modules and $N_i \in \tcH$-mod are  finite-dimensional.
Then, either
\enu[\textup{(}i\textup{)}]
\item $\Ext^1_B(M_1, M_2) \neq 0$, or
\item $\Hom_B(M_1, M_2) \neq 0$ and $\Ext^1_{\tcH}(N_1, N_2) \neq 0$.
\endenu
\end{lem}
\proof
We may assume that $M_1 \not\cong M_2$. 
Otherwise, 
$M_1 \cong M_2$ implies that $\Hom_B(M_1, M_2) \neq 0$.
Moreover, it follows from Proposition \ref{prop:exactness} that
$
\Ext^1_{B\wr \cH(d)}(M_1\wr N_1, M_2 \wr N_2) \cong \Ext^1_{\tcH}(N_1, N_2)
$,
and hence (ii) holds.

Next, since $\Ext_{B\wr \cH(d)}^1(M_1 \wr N_1, M_2 \wr N_2) \neq 0$,
there exists a non-split short exact sequence of ${B\wr \cH(d)}$-modules:
\[
0 \to M_2 \wr N_2 \overset{f}{\to} E \overset{g}{\to} M_1 \wr N_1 \to 0.
\]
Since $\tau^M(w\otimes m) = \sigma_w(m) \otimes h_w$,
Proposition \ref{prop:Hom}(a) holds.
That is, any ${B\wr \cH(d)}$-module homomorphism is a $\cB \otimes \tcH$-module homomorphism, and hence the following is a short exact sequence of $\cB \otimes \tcH$-modules:
\eq\label{eq:911SES2}
0 \to V_2 \overset{f}{\to} E \overset{g}{\to} V_1 \to 0,
\endeq
where $V_i := M_i^{\otimes d} \otimes N_i \in \cB \otimes \tcH$-mod.
We claim that the short exact sequence \eqref{eq:911SES2} is non-split, 
and thus $\Ext^1_{\cB \otimes \tcH}(V_1, V_2) \neq 0$.
Then, this lemma follows from the Kunneth formula, since
\eq
0 \neq \big(
\Ext^1_B(M_1, M_2)^{\otimes d} \otimes \Hom_{\tcH}(N_1, N_2)
\big)
\oplus
\big(
\Hom_B(M_1, M_2)^{\otimes d} \otimes \Ext^1_{\tcH}(N_1, N_2) 
\big).
\endeq
Now we prove the claim. Suppose that \eqref{eq:911SES2} splits. 
That is, there is a $\cB\otimes \tcH$-module homomorphism $s: V_1 \to E$ such that $g\circ s = \id_{V_1}$. We are done if $s$ is an ${B\wr \cH(d)}$-module homomorphism, or equivalently,
$H_i\cdot s(x) = s(H_i\cdot x)$ for all $x\in V_1$.
In other words, we want to show that the  map below is zero:
\eq
C_i: V_1 \to E,
\quad 
x \mapsto H_i\cdot s(x) - s(H_i\cdot x).
\endeq
First, observe that $gC_i(x) = g(H_i\cdot s(x)) - g\circ s(H_i\cdot x) = H_i\cdot g\circ s(x) - H_i\cdot x = 0$.
That is, $C_i(x) \in \ker g = \textup{im} f$, and hence the codomain of $C_i$ becomes $V_2$.
Next, since $C_i$ is $\cB$-linear, for any $b \in \cB$ and $x\in V_1$:
\eq
    C_i(b\cdot x) =H_i\cdot s(b\cdot x) - s(H_ib\cdot x)
    =(H_ib)\cdot s(x) - s(\sigma_i(b)H_i\cdot x)
    = \sigma_i(b) C_i(x).
\endeq
Thus, the codomain of $C_i$ becomes the $\cB$-module $V_2^{\sigma_i}$ with underlying space $V_2$ whose action is twisted by $\sigma_i$. 
Hence,
\eq
C_i \in \Hom_{\cB\otimes \tcH}(V_1, V_2^{\sigma_i}) 
\cong \Hom_\cB(M_1^{\otimes d}, (M_2^{\otimes d})^{\sigma_i}) 
\otimes \Hom_{\tcH}(N_1,N_2).
\endeq
Since $\Hom_\cB(M_1^{\otimes d}, (M_2^{\otimes d})^{\sigma_i}) 
\cong \Hom_\cB(M_1^{\otimes d}, M_2^{\otimes d}) \cong \Hom_B(M_1, M_2)^{\otimes d}$, 
it follows from Proposition \ref{prop:Hom}(b) that $C_i$ is a ${B\wr \cH(d)}$-module homomorphism. 
Using the quadratic relation $H_i^2 = R_i$, we obtain
\eq
\begin{split}
0 &= (H_i^2-R_i)\cdot s(x)
= H_i \cdot (H_i \cdot s(x)) - s(R_i \cdot x) 
\\
&= H_i \cdot (H_i \cdot s(x)) - s(H_i \cdot (H_i \cdot x)) 
\\
&= H_i\cdot C_i(x) + C_i (H_i \cdot x)
\\
&= 2 C_i (H_i \cdot x).
\end{split}
\endeq
Since $\textup{char}(K) \neq 2$, $C_i$ is the zero map. The claim is proved.
\endproof

\section{Important Examples of Wreath Modules}\label{sec:WMQWP}
In this section, we provide details to the construction of wreath modules for several quantum wreath products. 

\subsection{Wreath Index Set and Dominance Order} \label{sec:indexset}
Assume that $I_B$ is a poset that indexes a certain set of modules for the base algebra $B$. 
Typically, $I_B$ indexes the complete set of non-isomorphic finite-dimensional simple $B$-modules, but we want to keep the flexibility here.
The goal here is to give a partially-ordered index set $\Omega$ for simple modules (up to isomorphism) of the quantum wreath product $ B \wr \cH(d)$. 

Recall that $\Pi_m$ is the set of partitions of $m$, $\Pi = \bigcup_{m \geq 0} \Pi_m$,
and  ${}^\#\lambda := \sum_{i} \lambda_i$.
Denote the set of $I_B$-partitions of $d$ by
\eq
\textstyle
\Pi_{I_B}(d) := \left\{
\bblambda: I_B \to \Pi
~\middle|~
\sum_{\nu \in I_B} {}^\#\bblambda(\nu) = d
\right\}.
\endeq
Write $\Omega:=  \Pi_{I_B}(d) $ for brevity.
Recall the characteristic function $e_\lambda^\nu \in \Omega$ defined in \eqref{def:charfcn}.
\begin{Def}[Dominance Order]
\label{def:domPO}
The set $\Pi_{I_B}(d)$ has a partial order $\leq$ induced from $(I_B,\leq)$ given by $\bblambda \leq \bblambda'$ if and only if 
\eq\label{def:POIA}
\sum_{j \leq k}\bblambda(\nu)_j
+
\sum_{\gamma < \nu \in I_B} {}^\#\bblambda(\gamma)  
\ \leq \ 
\sum_{j \leq k}\bblambda'(\nu)_j
+
\sum_{\gamma < \nu \in I_B} {}^\#\bblambda'(\gamma)
\qquad
(\textup{for all}
~ \nu\in I_B, 
~k\geq 1).
\endeq
\end{Def}
Therefore, suppose that $\lambda \in I_B$ and $\nu, \nu' \in \Pi_d$, the following implications are immediate:
\eq 
 \lambda < \lambda' \in I_B \Rightarrow  e_\lambda^\nu < e_{\lambda'}^{\nu'},
\quad
\textup{and}
\quad
\nu \triangleright \nu' \in \Pi_d \Rightarrow e_\lambda^\nu < e_{\lambda}^{\nu'}.
\endeq
\lem\label{lem:PO}
Let $\mu = (\mu_1, \dots, \mu_r) \vDash d$. Suppose that for all $1\leq i \leq r$,
$\lambda^{(i)} \leq \gamma^{(i)}$ in $I_B$,
$\nu^{(i)}, \delta^{(i)} \vdash \mu_i$,
and
\[
\nu^{(i)} \triangleleft \delta^{(i)}
\quad
\textup{if}
\quad
\lambda^{(i)} = \gamma^{(i)}.
\]
Then, $\sum_i e_{\lambda^{(i)}}^{\nu^{(i)}} < \sum_i e_{\gamma^{(i)}}^{\delta^{(i)}}$ in $\Omega$.
\endlem
\proof
It follows from the fact that for all $\nu \in I_B$,  
\eq
\{ \gamma \in I_B ~|~ \gamma < \nu\} \cap \{\lambda^{(1)}, \dots, \lambda^{(r)}\} \subseteq 
\{ \gamma \in I_B ~|~ \gamma < \nu\} \cap \{\gamma^{(1)}, \dots, \gamma^{(r)}\},
\endeq
and hence \eqref{def:POIA} must hold.
\endproof

Sometimes, it is useful to fix an enumeration $\{\lambda^{(1)}, \dots, \lambda^{(s)} \}$ of $I_B$ that is compatible with the partial order $\leq$, i.e.,  $\lambda^{(i)}<\lambda^{(j)}$ in $I_B$ then $i<j \in \ZZ$.
Given $\bblambda:I_B\to \Pi$ with support supp$(\bblambda) = \{\lambda^{(k_1)}, \dots, \lambda^{(k_r)}\}$ where $1\leq k_1 < \dots < k_r \leq s$,
define a composition $\mu = \mu(\bblambda) \vDash d$ by $\mu_i := {}^\#\bblambda(\lambda^{(k_i)})$, and a multipartition
\eq
\overline{\bblambda} := (
\underset{\mu_1\textup{ times}}{\underbrace{\lambda^{(k_1)}, \dots, \lambda^{(k_1)}}}, 
\underset{\mu_2\textup{ times}}{\underbrace{\lambda^{(k_2)}, \dots, \lambda^{(k_2)}}}, 
\dots, 
\underset{\mu_r\textup{ times}}{\underbrace{\lambda^{(k_r)}, \dots, \lambda^{(k_r)}}} 
),
\endeq
equivalently, $\overline{\bblambda}$ is a map $\{1,\dots,d\} \to \Pi$ such that $\overline{\bblambda}(i) := \lambda^{(k_j)}$ where $k_1 + \dots + k_{j-1} < i \leq k_1 + \dots + k_j$.
%
\exa\label{ex:lambdabar}
Let $\lambda^{(1)}, \lambda^{(2)}, \lambda^{(3)}$ be distinct elements in $I_B$ such that $\lambda^{(1)} < \lambda^{(i)}$ for $i=2,3$, and let 
\[
\bblambda :=e_{\lambda^{(2)}}^{(1)} + e_{\lambda^{(3)}}^{(1)},
\quad
\bbmu := e_{\lambda^{(1)}}^{(1,1)}.
\]
Then, supp$(\bblambda) = \{ \lambda^{(2)}, \lambda^{(3)}\}$, supp$(\bbmu) = \{\lambda^{(1)}\}$, and
\[
\overline{\bblambda} = (\lambda^{(2)}, \lambda^{(3)}),
\quad
\overline{\bbmu} = (\lambda^{(1)}, \lambda^{(1)}).
\]
It then follows from Lemma \ref{lem:PO} that $\bblambda > \bbmu$, since
 $\lambda^{(i)} > \lambda^{(1)}$ for $i=2,3$.
\endexa
\exa\label{ex:AKindex}
For the (modified) Ariki--Koike algebra over $K$ or the group algebra over $\CC$ for the complex reflection group $C_m \wr \Sigma_d$ of type $G(m,1,d)$, the irreducible modules over $B$ are indexed by $I_B = \{1, \dots, m\}$ using the natural total order.
Recall that the base algebras are 
\[
B = \frac{K[t]}{\prod_{i=1}^m (t - u_i)},
\quad
\textup{or}
\quad 
B = \CC C_m := \frac{\CC[t]}{(t^m-1)},
\] 
respectively, where $u_i \in K$ are elements such that $\prod_{i>j} (u_i - u_j)$ is invertible.  
For any $\lambda \in I_B$, the corresponding simple module is the one dimensional module
$S^\lambda$ on which $t$ acts by $u_i$ and by $\xi^i$, respectively, where $\xi \in \CC^\times$ is a primitive $m$th root of unity. 

Then, each $\bblambda \in \Omega$ can be identified as a multipartition  $\bblambda \equiv (\bblambda(1), \bblambda(2), \dots, \bblambda(m))$ of $d$.
Moreover, the partial order \eqref{def:POIA} is precisely the usual dominance order on the set of multipartitions, which is defined using the lexicographical order.
For example, when $m=2=d$, the partial order on $\Omega$ is given by
\eq
\left(\begin{array}{l} 2 \mapsto\varnothing \\ 1\mapsto (1,1)\end{array}\right)
\unlhd
\left(\begin{array}{l} 2 \mapsto \varnothing\\ 1\mapsto (2)\end{array}\right)
\unlhd
\left(\begin{array}{l} 2 \mapsto (1)\\ 1\mapsto (1)\end{array}\right)
\unlhd
\left(\begin{array}{l} 2 \mapsto (1,1)\\ 1\mapsto \varnothing\end{array}\right)
\unlhd
\left(\begin{array}{l} 2 \mapsto (2) \\ 1\mapsto \varnothing\end{array}\right),
\endeq
which is equivalent to that 
$(\varnothing, (1,1)) \unlhd (\varnothing, (2)) \unlhd ((1), (1)) \unlhd ((1,1), \varnothing) \unlhd ((2), \varnothing)$.
\endexa
\subsection{Examples of Wreath Index Sets}
\exa\label{ex:PQWPindex}
Let $B = F[x]$ or $F[x^{\pm1}]$ be the ring of (Laurent) polynomials.
Consider any  (infinite-dimensional) $B$-module $T$ with $F$-basis $\{v_i \}_{i \in \ZZ}$ such that $x$ acts by $x\cdot v_i := v_{i+n}$ for some fixed  $n \in \ZZ_{>0}$.

For the purpose of understanding certain Kashiwara-Miwa-Stern tensor space $T^{\otimes d}$ (see Section \ref{sec:WMPQWP}), we need to consider the set of all simple submodules of $T$.
Each simple submodule is of the form $V_j := \textup{Span}_F\{v_k ~|~ k \in  j +n\ZZ\} \subseteq T$ for $1\leq j \leq n$.

Thus, we may set $I_B = \{1 < \dots < n\}$, and then identify
$\Omega$ with the set of multipartitions of $d$ of $n$ parts, via $\bblambda \mapsto (\bblambda(1), \dots, \bblambda(n))$.
We will see that $\Omega$ is useful in describing all simple submodules of $T^{\otimes d}$. 
\endexa
\exa\label{ex:Huindex}
For the Hu algebra or the group algebra of $\Sigma_m \wr \Sigma_2$, 
the base algebras are the Hecke algebra $B= \cH_q(\Sigma_m)$ and the group algebra $K\Sigma_m$, respectively.
The irreducibles over  $B$ are in bijection with the set of Specht modules, and hence are indexed by the set $I_B =\Pi_m$ of partitions of $m$. 

In the special case $m=2$,  the  poset structure $I_B =\{(1,1) < (2)\}$  induces a partial order on $\Omega$ given by
\eq\label{eq:HuPoset}
\left(\begin{array}{l}  (2)\mapsto\varnothing \\ (1,1) \mapsto (1,1)\end{array}\right)
\unlhd
\left(\begin{array}{l} (2) \mapsto \varnothing\\ (1,1)\mapsto (2)\end{array}\right)
\unlhd
\left(\begin{array}{l} (2) \mapsto (1)\\ (1,1)\mapsto (1)\end{array}\right)
\unlhd
\left(\begin{array}{l} (2) \mapsto (1,1)\\ (1,1)\mapsto \varnothing\end{array}\right)
\unlhd
\left(\begin{array}{l} (2) \mapsto (2) \\ (1,1)\mapsto \varnothing\end{array}\right).
\endeq
Note that the decorated orbits in \cite{Hu02} are identified as
$[\lambda, \lambda]_+ \equiv e_\lambda^{(2)}$, $[\lambda, \lambda]_- \equiv e_\lambda^{(1,1)}$, and $[\lambda, \mu] \equiv e_\lambda^{(1)} + e_\mu^{(1)}$, and hence 
 \eqref{eq:HuPoset} becomes 
\[
 [(1,1), (1,1)]_- \unlhd [(1,1), (1,1)]_+ \unlhd [(1,1), (2)] \unlhd [(2), (2)]_- \unlhd [(2), (2)]_+.
\]
\endexa

\subsection{Transpose Map on the Index Set}\label{sec:transpose}
Now, assume that the poset $I_B$ admits the notion of transposition, i.e., an order-reversing involution $t: \nu \mapsto \nu^t$.
Then, we define the transpose map on $\Omega$ via $\bblambda \mapsto \bblambda^t$, where
\eq\label{def:t}
\bblambda^t : I_B \to \Pi,
\quad
\nu \mapsto \bblambda(\nu^t)^t
\endeq
It is clear that $t$ is an involution.
\exa
Consider the Hu algebra $B \wr \cH(2)$ with  base algebra $B = \cH_q(\Sigma_d)$.
The Specht modules of the form $S^\bblambda = S(\lambda^{(1)}, \lambda^{(2)})$ for $\lambda^{(1)} \neq \lambda^{(2)} \in \Pi_m$ corresponds to the multipartition $\bblambda$ supported on $\{\lambda^{(1)}, \lambda^{(2)}\}$ with the same value $(1)$.
Thus, $\bblambda^t$ is supported on $\{(\lambda^{(1)})^t, (\lambda^{(2)})^t\}$ with the same value $(1)$.
That is, the  induced transpose map $S^\bblambda \mapsto S^{\bblambda^t}$ takes $S(\lambda^{(1)}, \lambda^{(2)})$ to $S((\lambda^{(1)})^t, (\lambda^{(2)})^t)$ for $\lambda^{(1)} \neq \lambda^{(2)} \in \Pi_m$.

The other kind of Specht modules $S^\bblambda = S(\lambda, \lambda)_\pm$ correspond to the multipartitions 
\eq 
\bblambda: \nu \mapsto
\begin{cases}
(2) &\tif \nu = \lambda;
\\
\varnothing &\textup{otherwise},
\end{cases}
\quad\textup{and}\quad
\bblambda: \nu \mapsto
\begin{cases}
(1,1) &\tif \nu = \lambda;
\\
\varnothing &\textup{otherwise},
\end{cases}
\endeq
respectively. 
Hence, the induced transpose map swaps the two Specht modules $S(\lambda,\lambda)_\pm$ and $S(\lambda^t,\lambda^t)_\mp$ for any $\lambda \in \Pi_m$. 
\endexa

\subsection{Wreath Modules over Ariki--Koike Algebras}\label{sec:WMAK}
Following Section \ref{sec:mAK}, let $B\wr \cH(d)$ be the modified Ariki--Koike algebras, and let
$\tcH_d^{\chi}$ be the usual Hecke algebra with generators $h_1, \dots h_{d-1}$ subject to the type A braid relations as well as the quadratic relation $h_i^2 - \SK h_i - \RK = 0$, where $\SK = v-v^{-1}, \RK = 1 \in K$.

We obtain the following wreath modules as a consequence of Theorem~\ref{thm:MwrN}.
\begin{cor}\label{cor:MwrNAK}
Let $B = \frac{K[t]}{\prod_{i=1}^m(t-u_i)}$, and let $\gamma \in \{v, -v^{-1}\}$.
\begin{enua}
\item
Suppose that $M \in B$-mod is a simple module, and $N$ is an $\tcH^\chi_d$-mod.  
Then, $M\wr N \in B \wr\cH(d)$-mod, on which the action is given by, for $b\in B^{\otimes 2}, m \in M^{\otimes 2}$, $y\in N$:
\[
b \cdot (m \otimes y) = (b \cdot m) \otimes y,
\quad
H_1 \cdot (m \otimes y) =  \gamma (m \otimes \cdot y).
\]
\item
Suppose that
$M_1, \dots, M_r \in B$-mod are simple modules,
and that each $N_i$ is an $\tcH^{\chi}_{\mu_i}$-module. 
Then, 
\[
(M_1 \wr N_1) \otimes (M_2 \wr N_2) \otimes  \dots \otimes (M_r \wr N_r) 
\equiv
M_1^{\otimes \mu_1} \otimes \dots \otimes M_r^{\otimes \mu_n} \otimes N_1 \otimes \dots \otimes N_r
\] 
is a module over the parabolic subalgebra $B \wr \cH(\mu)$ of the modified Ariki--Koike algebra.
\end{enua}
\end{cor}
Following Example \ref{ex:AKindex}, consider the set of simple modules  $\{K_\lambda ~|~ \lambda \in I_B \}$ over the finite dimensional algebra $B = \frac{K[t]}{\prod_{i=1}^m(t-u_i)}$, where $I_B = \{1, 2, \dots, m\}$.

For any $\bblambda \in \Omega$, its support is of the form 
$\{\lambda^{(1)}, \dots, \lambda^{(r)}\}$ for some distinct $\lambda^{(i)} \in I_B$.
Write $\mu_i := {}^\#\lambda^{(i)}$ and $\nu^{(i)} := \bblambda(\lambda^{(i)})$.
For each $i$, we set $M_i := K_{\lambda^{(i)}}$ to be the corresponding simple module.

Next, we can set  $N_i = S^{\nu^{(i)}}_{\chi}$ to be the Specht module
over the Hecke algebra $\tcH_\chi^d$. Consequently, Corollary \ref{cor:MwrNAK} implies that the following is a module over the Ariki--Koike algebra:
\begin{eq}
L(\bblambda) := \ind_{\mu}^d  \big( 
(K_{\lambda^{(1)}} \wr S_{\chi}^{\nu^{(1)}}) 
\otimes \dots \otimes
(K_{\lambda^{(r)}} \wr S_{\chi}^{\nu^{(r)}}) 
\big).
\end{eq}
It is known (see \cite{SS05}) that the following set gives the complete set of non-isomorphic simple modules over the modified Ariki--Koike algebras:
\[
\{L(\bblambda) ~|~ \bblambda \in \Omega\}
=
\left\{
\ind_{\mu}^{d} \big( \bigotimes\nolimits_i (K_{\lambda^{(i)}} \wr S_{\chi}^{\nu^{(i)}}) \big)
~\middle|~
\mu \vDash d, 
\  \nu^{(i)} \vdash \mu_i,
\ 1\leq \lambda^{(i)} \leq m
\right\}.
\]

\subsection{Wreath Modules over Polynomial Quantum Wreath Products}\label{sec:WMPQWP}
Following \cite{LM25}, let $B\wr \cH(d)$ be a polynomial quantum wreath product.
That is, $B = F[x^{\pm 1}]$ or $F[x]$ is the ring of (Laurent) polynomials over a $K$-algebra $F$. 

There is an element $\beta = \sum_{0\leq i,j \leq 1} \Delta^{ij} (x^i\otimes x^j) \in B\otimes B$ for some weak Frobenius elements $\Delta^{ij}$, i.e., $\Delta^{ij} (f\otimes g) = (g\otimes f)\Delta^{ij}$ for all $f,g \in F$.
The parameters are given by
\eq
S = \Delta^{10} - \Delta^{01}, 
\quad
R \in (Z(F\otimes F))^{\Sigma_2},
\quad
\sigma:a\otimes b \mapsto  b\otimes a,
\quad
\rho = \partial^\beta,
\endeq
where $\partial^\beta$ is the twisted Demazure operator given by $\rho(a\otimes b) = \frac{\beta (a\otimes b) - (b\otimes a) \beta}{x\otimes 1 -1\otimes x}$ for all $a,b \in B$.
We assume further that $B\wr\cH(d)$ has a Kashiwara-Miwa-Stern (KMS) tensor space, which is a consequence by making the following assumptions:
\begin{itemize}
\item[(C1)] There is an element $\alpha \in F\otimes F$ such that $(H+\alpha)(H-\alphab) = 0$, where $\alphab := \sigma(\alpha) + S$;
\item[(C2)] Elements in $F\otimes F$ commute with both $\alpha$ and $\beta$;
\item[(C3)] The element $\alpha(x\otimes 1 - 1\otimes x) +\beta$ is not a zero divisor.
\end{itemize}
By a slight abuse of notation, we consider twisted Hecke algebras over $F^{\otimes d}$ instead of over $K$ within this subsection.
We obtain the following wreath modules as a consequence of Theorem~\ref{thm:MwrN}.

\begin{cor}\label{cor:MwrNAHA}
Suppose that $M, M_1, \dots, M_r \in F[x^{\pm1}]$-mod are cyclic, i.e., of the form $F[x^{\pm1}] \cdot v$ for some $v$.
\begin{enua}
\item 
If $N$ is a module over a twisted Hecke algebra $\tcH^{\chi}_d$ satisfying \eqref{asm:twistchi}.
Then, $M\wr N \in B \wr\cH(d)$-mod, on which the action is given by, for $b\in F[x^{\pm1}]^{\otimes d}$, $y\in N$:
\[
b \cdot (v \otimes y) = (b \cdot v) \otimes y,
\quad
H_i \cdot ((b \cdot v) \otimes y) =  (\sigma_i(b) \cdot v) \otimes (h_1 \cdot y)
+(\rho_i(b) \cdot v) \otimes y.
\]
\item
For $1\leq i\leq r$, if each $N_i$ is a module over $\tcH^{\chi_{i}}_{\mu_i}$, which is a twisted Hecke algebra satisfying \eqref{asm:twistchi} with respect to $M=M_i$.
Then, 
\[
(M_1 \wr N_1) \otimes (M_2 \wr N_2) \otimes  \dots \otimes (M_r \wr N_r) 
\equiv
M_1^{\otimes \mu_1} \otimes \dots \otimes M_r^{\otimes \mu_n} \otimes N_1 \otimes \dots \otimes N_r
\] 
is a module over the parabolic subalgebra $B \wr \cH(\mu)$ of $B \wr \cH(d)$.
\end{enua}
\end{cor}
Consider the KMS tensor space $T_n^{\otimes d}$, where
$T_n := \bigoplus_{i \in \ZZ} F v_i$, on which $x^{\pm1}$ acts by $x^{\pm1} \cdot v_{i} = v_{i\pm n}$.
Following Example~\ref{ex:PQWPindex}, consider the set $\{V_\lambda ~|~ \lambda \in I_B\}$ of simple submodules of $T_n$, where $I_B = \{1, \dots, n\}$.

Note that each $\bblambda \in \Omega$ has support of the form 
$\{1\leq k_1 < \dots < k_r \leq n\}$ for some $r$.
Here, we consider the subset $\Omega'$ of $\Omega$ consisting of those $\bblambda$ such that $\bblambda(\lambda)$ is a one-row partition for all $\lambda \in I_B$.
Hence, we can write $\bblambda(k_i) = (\mu_i)$ for some integer $\mu_i$.
For each $i$, set $M_i = V_{k_i}$.
Note that any $m \in V_{k_i}$ is of the form $m = P \cdot (v_{k_i} \otimes v_{k_i})$ for some (Laurent) polynomial $P(x_1, x_2)$, where $x_1 := x \otimes 1, x_2 := 1\otimes x$.
Let $\chi_i := \chi_{V_{k_i}}$ be the central character given by $\chi_i( P \cdot (v_{k_i} \otimes v_{k_i})) = P(1,1)$, i.e., the evaluation of $P$ at $x=1$.

In other words, each twisted Hecke algebra $\tcH^{\chi_i}_{\mu_i}$ has quadratic relations
\eq
0 = (h_j + \alpha_j) (h_j - \alphab_j),
\quad
\textup{for all}
\quad
1\leq j \leq \mu_i-1.
\endeq
We want $N_i = F^{\otimes \mu_i}$ be the ``trivial module'' on which $h_i$ acts by multiplication by $\alphab_i$.
By convention, it means $N_i = S_{\chi_i}^{(\mu_i)}$.
Therefore, the KMS tensor space can be decomposed as $T_n^{\otimes d} = \bigoplus_{\bblambda \in \Omega'} T_\bblambda$ where 
\eq
T_\bblambda :=\ind_\mu^d
\big(
(V_{k_1} \wr S_{\chi_1}^{(\mu_1)})
\otimes \dots \otimes 
(V_{k_r} \wr S_{\chi_r}^{(\mu_r)})
\big).
\endeq
Note that $T_\bblambda \cong  \ind_{\mu}^d
\big(
(V_1 \wr S_{\chi_1}^{(\mu_1)})
\otimes \dots \otimes 
(V_r \wr S_{\chi_r}^{(\mu_r)})
\big)$ as $B \wr \cH(d)$-modules, and hence each summand of such form occurs ${\mu \choose \mu_1; \mu_2; \dots; \mu_r}$ times in $T_{n}^{\otimes d}$.

\subsection{Wreath Modules over Hu Algebras}\label{sec:WMHu}
Following Section \ref{sec:gHu}, let $B =\cH_q(\Sigma_m)$,
and let $B \wr \cH(d)$ be the generalized Hu algebra.
We obtain the following wreath modules as a consequence of Theorem~\ref{thm:MwrN}.

\begin{cor}\label{cor:MwrNHu}
Suppose that $M, M_1, \dots, M_r \in \cH_q(\Sigma_m)$-mod such that the element $z_m \in \cH_q(\Sigma_m) \otimes \cH_q(\Sigma_m)$ acts on $M\otimes M, M_i\otimes M_i$ by  multiplication by $\chi(Z_m)$ and $\chi_i(z_m)$, respectively.
\begin{enua}
\item
If $N$ is a module over $\tcH^{\chi}_d$, which is a twisted Hecke algebra satisfying \eqref{asm:twistchi}
Then, $M\wr N \in B \wr\cH(d)$-mod, on which the action is given by, for $b\in \cH_q(\Sigma_m)^{\otimes 2}, m \in M^{\otimes 2}$, $y\in N$:
\[
b \cdot (m \otimes y) = (b \cdot m) \otimes y,
\quad
H_1 \cdot (m \otimes y) =  m \otimes (h_1 \cdot y).
\]
\item
For $i = 1, \dots, r$, if each $N_i$ is a module over  $\tcH^{\chi_{i}}_{\mu_i}$, which is a twisted Hecke algebra satisfying \eqref{asm:twistchi} with respect to $M = M_i$. 
Then, 
\[
(M_1 \wr N_1) \otimes (M_2 \wr N_2) \otimes  \dots \otimes (M_r \wr N_r) 
\equiv
M_1^{\otimes \mu_1} \otimes \dots \otimes M_r^{\otimes \mu_n} \otimes N_1 \otimes \dots \otimes N_r
\] 
is a module over the parabolic subalgebra $\cH_q(\Sigma_m) \wr \cH(\mu)$ of the generalized Hu algebra.
\end{enua}
\end{cor}
Following Example~\ref{ex:Huindex}, consider the set of Specht modules $\{S_m^\lambda ~|~ \lambda \in I_B\}$ over the Hecke algebra $B = \cH_q(\Sigma_m)$, where $I_B = \Pi_m$.

We need to be careful when $I_B$ is not a totally ordered set.
We may fix a enumeration $I_B = \{\lambda^{(1)}, \lambda^{(2)}, \lambda^{(3)}, \dots\}$ that is compatible with the dominance order on $I_B = \Pi_m$, i.e., if $\lambda^{(i)} \lhd \lambda^{(j)}$ then $i<j$.
Thus,  each $\bblambda \in \Omega$ has support of the form 
$\{\lambda^{(k_i)} ~|~  k_1 < k_2 < \dots < k_r \}$ for some $r$.
Write 
$|\bblambda| := ({}^\#\lambda^{(k_1)}, \dots , {}^\#\lambda^{(k_r)}) \vDash d$.
For example, if $\bblambda = e_{\lambda^{(1)}}^{(2)} + e_{\lambda^{(3)}}^{(1)} +e_{\lambda^{(4)}}^{(1,1)}$, then $|\bblambda| = (2,1,2)$.

For each $i$, we set $M_i := S^{\lambda^{(k_i)}}_m$ to be the corresponding Specht module for $\cH_q(\Sigma_m)$,
and let $\chi_i$ be the twist determined by $\chi_i(S) = 0$ and $\chi_i(R) = f_{\lambda^{(k_i)}}$.

Next, let
$\nu^{(i)} := \bblambda(\lambda^{(k_i)})$. 
We pick  $N_i = S^{\nu^{(i)}}_{\chi_i}$ to be the Specht module as in Section \ref{sec:tHASpecht}.
Consequently, Corollary \ref{cor:MwrNHu} implies that the following vector space which we call the Specht module is indeed a module over the generalized Hu algebra:
\eq\label{def:SpechtgHu}
S^\bblambda := \ind_{|\bblambda|}^d  \big( 
(S_{m}^{\lambda^{(k_1)}} \wr S_{\chi_1}^{\nu^{(1)}}) 
\otimes
\dots
\otimes
(S_{m}^{\lambda^{(k_r)}} \wr S_{\chi_r}^{\nu^{(r)}})
\big).
\endeq
Note that $S^\bblambda$ does not depend on the choice of enumeration of $I_B$.

In other words, when $d=2$, the set $\{S^\bblambda ~|~ \bblambda \in \Omega\}$ recovers the set of Specht modules over the Hu algebra as in Corollary \ref{cor:gHuIrrep}.

\subsection{Wreath Modules over Yokonuma--Hecke Algebras}
{
Let $m := q-1$ where $q$ is a prime power, 
fix a primitive $m$th root of unity $\xi \in \CC$,
and let $B = \CC[t]/(t^m-1)$.
Consider the Yokonuma--Hecke algebra for $\Sigma_d$ as a quantum wreath product $B \wr \cH(d)$ with the following choices of parameters:
\begin{equation}
    S = (q-1)e, 
    \quad
    R = q(1\otimes1),
    \quad
    \sigma = \textup{flip},
    \quad
    \rho = 0,
\end{equation}
where $e:= \frac{1}{m}\sum_{i=1}^m t^i\otimes t^{-i}$ is an idempotent in $B\otimes B$.
    
The simple $B$-modules are $\{\CC_1, \dots, \CC_m\}$, where $\CC_i$ is the 1-dimensional $B$-module on which $t$ acts by multiplication by $\xi^i \in \CC^\times$.
Set $M= \CC_a$ for some $1\leq a \leq m$.
For any $1\leq b \leq m$, the idempotent $e$ acts on $\CC_a \otimes \CC_b$ by
\begin{equation}
    \frac{1}{m} \sum_{i=1}^m \xi^{i(a-b)} = \begin{cases}
        1&\tif a=b;
        \\
        0&\tif a\neq b.
    \end{cases}
\end{equation}
Consider $M\otimes M = \CC_a \otimes \CC_a$, on which $S$ and $R$ act by multiplication by $q$ and $q-1$, respectively.
Thus, the twisted Hecke algebra $\tcH^{\chi_M}_d$ is the Hecke algebra $\cH_q(\Sigma_d)$.
Recall that $S_d^{\nu} \in \cH_q(\Sigma_d)$-mod is the Specht module corresponding to the partition $\nu \vdash d$.
\begin{prop}
The wreath module
    $\CC_a \wr S_d^{\nu}$ is a module over the Yokonuma--Hecke algebra $B \wr \cH(d)$.
\end{prop}
\begin{proof}
Define $\sigma^M = \id$ and $\rho^M = 0$ to be the desired $\CC$-linear endomorphisms on $M\otimes M$ in the structure map $\tau^M$.
Here, \eqref{def:M2}--\eqref{def:taubr3} hold almost automatically.
Therefore, $M = \CC_a$ is a \based module, according to Theorem \ref{thm:basedmod}.

Next, since the twisted Hecke algebra $\tcH^{\chi_M}_d$ is the Hecke algebra $\cH_q(\Sigma_d)$, 
 $\CC_a \wr S_d^{\nu}$ is a $B\wr \cH(k)$-module by Theorem~\ref{thm:MwrN}.
\end{proof}

\begin{cor}
    The complete set of simple modules of the Yokonuma--Hecke algebra is given by
\[
\left\{
L(\bblambda) :=
\ind_{B \wr\cH(\mu)}^{B \wr \cH(d)} \big( \bigotimes\nolimits_i (\CC_{a_i} \wr S_{\mu_i}^{\nu^{(i)}}) \big)
\right\}_{\bblambda\in I_A},
\quad
I_A := \left\{
\bblambda: \{1, \dots, m\} \to  \Pi \middle|\sum\nolimits_{1\leq j \leq m} {}^\#\bblambda(j) = d
\right\},
\]
where the support of $\bblambda$ is $\{a_1, \dots, a_r\} \subseteq \{1, \dots, m\}$ for some $r$,
$\nu^{(i)} := \bblambda(a_i) \vdash \mu_i \leq d$ for each $1\leq i \leq r$,
and $\mu = (\mu_1, \dots, \mu_r) \vDash d$ is a composition.
\end{cor}
\begin{proof}
This is done by paraphrasing the classification result \cite[Theorem 3.7]{CP14} due to Chlouveraki and Poulain d'Andecy. 

Our $\bblambda$ is identified with the multipartition referred as $\boldsymbol{\lambda} = (\lambda^{(0)}, \dots, \lambda^{(m-1)})$ of $d$ therein by
\[
\lambda^{(j)} = \begin{cases}
\bblambda(a_i) &\tif j = a_i;
\\
\varnothing &\textup{otherwise}.
\end{cases}
\]
By comparing the action on Young's seminormal form, one obtains immediately that our $L(\bblambda)$ is isomorphic to the simple module $V_{\boldsymbol{\lambda}}$ therein.
\end{proof}
}

\subsection{Cuspidal Modules are Wreath Modules} \label{sec:cuspidal}
{
Observe that construction of the wreath modules admits the following generalization.
Let $M_1, \dots, M_d$ be $B$-modules.
By a slight abuse of notation, within this subsection set
$\cM = M_1 \boxtimes \dots \boxtimes M_d$,
$\alpha_M$ be the structure map $B^{\otimes d} \otimes\cM \to \cM$,
$\sigma^{M}_i, \rho^{M}_i$ be arbitrary $K$-linear maps in $\End_K(\cM)$ which are necessarily determined by the rank two maps $\sigma^M$ and $\rho^M$,
$\tau^M$ be the map defined by the same formula \eqref{def:tauaction}, and let
\begin{equation}\label{def:newWM}
    (bH_w) \cdot (m\otimes n) := (\alpha_M \otimes \alpha_N)(b\otimes \tau^M(w\otimes m) \otimes n).
\end{equation}
Then, the proof of Theorem \ref{thmA} generalizes verbatim to the following.
\begin{prop}\label{prop:thmA'}
Suppose that each $S_i$ (resp. $R_i$) acts on $\cM$ by the same $\chi(S)$ (resp. $\chi(R)$) in $K$, and $N \in \tcH^\chi_d$-mod.
Then, \eqref{def:newWM} defines a $B\wr \cH(d)$-module structure on $\cM \otimes N$ if and only if the following are true:
\begin{itemize}
    \item \eqref{def:M2} and \eqref{def:M4} hold.
    \item \eqref{def:taubr1}--\eqref{def:taubr3} hold additionally if $d\geq 3$.    
\end{itemize}
\end{prop}
We write $\cM \otimes N = M_1 \tilde{\wr} N$ if all $M_i$ have the same underlying space. 
Now, let $B \wr \cH(d)$ be the degenerate affine Hecke algebra of type A. To be precise, $B = \ZZ[x]$, $S = 0$, $R=1$, $\sigma = $flip, and $\rho$ is determined by $\rho(x_1) = -1$.

The simple $B$-modules are given by $\{\CC_a ~|~a \in \CC\}$, where $\CC_a$ is 1-dimensional and on which $x$ acts by multiplication by $a$.
It is well-known that the simple $B \wr \cH(d)$-modules can be constructed via parabolic inductions on the cuspidal modules of the form $\CC_a\boxtimes \dots \boxtimes\CC_{a+k-1}$ for some $a\in \CC, 1\leq k \leq d$. That is, the $B \wr \cH(k)$-module structure is given by
\begin{equation}   \label{action:segmentdAHA}
(H_w x_j) \cdot 1^{\otimes k} := (a+j-1)1^{\otimes k}
\quad
\textup{for all }
w\in\Sigma_k,
\ 1\leq j \leq k.
\end{equation}
\begin{prop}
As modules of the degenerate affine Hecke algebra, 
the cuspidal module $\CC_a\boxtimes \dots \boxtimes\CC_{a+k-1}$ is isomorphic to $\CC_a \tilde{\wr} S^{(k)}$ where $S^\nu$ is the Specht module of $\CC\Sigma_k$ corresponding to the partition $\nu$.
\end{prop}
\begin{proof}
First, we apply Proposition~\ref{prop:thmA'} to show that $\CC_a \wr S^{(k)}$ admits a $B\wr\cH(k)$-module structure.
    Let $M = \CC_a$. Consider the following $B^{\otimes k}$-module $M^{\otimes k}$ with a twist on each tensor factor:
    \[
    x_1\cdot m = a m, 
    \quad 
    x_2 \cdot m = (a+1) m,
    \quad
    \dots
    \quad
    x_k\cdot m = (a+k-1) m.
    \]
Define $\sigma^M_i = 0$ and $\rho^M_i = \id$ for all $i$, and hence $H_i$ indeed acts trivially on the proposed wreath module $\CC_a \wr S^{(k)}$, provided \eqref{def:M2}--\eqref{def:taubr3} hold.

Note that the second equality in \eqref{def:M2} reduces to the rank two verification:
\begin{equation}\label{eq:dAHAM2}
b\cdot m = \sigma(b)\cdot m + \rho(b)\cdot m,
\quad
\textup{for all }
b \in \ZZ[x_1, x_2],
\ m \in M\otimes M.
\end{equation}
Since $\rho(x^i_1x^j_2) = (x_1x_2)^l\rho(x^{i-l}_1x^{j-l}_2)$ where $l := \min\{i,j\}$, one may factor out $(x_1x_2)^l$ from both sides of \eqref{eq:dAHAM2}.
That is, it suffices to prove the special case when $b = x^i_1$ or $x^i_2$, which follows from a direct calculation.

The rest of the equalities in \eqref{def:M2}--\eqref{def:taubr3} follow trivially.
Next, since $S = 0$, $R=1$, the twisted Hecke algebra $\tcH_k^{\chi_M}$ is precisely to the group algebra $\CC \Sigma_k$,
and hence $\CC_a \tilde{\wr} S^{(k)}$ is a $B\wr \cH(k)$-module, on which the action agrees with the one given in \eqref{action:segmentdAHA}.
Therefore, $\CC_a \tilde{\wr} S^{(k)} \cong \CC_a\boxtimes \dots \boxtimes \CC_{a+k-1}$.
\end{proof}
A similar argument proves that the cuspidal module $\CC_a \boxtimes \dots \boxtimes \CC_{aq^{k-1}}$ for the affine Hecke algebra is isomorphic to $\CC_a \tilde{\wr} S^{(k)}_k$, where $S^{\nu}_k$ is the Specht module for the Hecke algebra $\cH_q(\Sigma_k)$.
}

\section{Technical proofs} \label{sec:app}

\subsection{Supplements of Proof of  Theorem \ref{thm:basedmod}}\label{sec:abbrev}

For simplicity, in several proofs we will abbreviate tensor products by juxtapositions, e.g., $\id \otimes \tau^M$ can be expressed as $\cB W \cM  \overset{\id\tau^M}{\longrightarrow} \cB \cM \tcH$.
Additionally, we use a combination of a unique outward arrow as well as underbraces to abbreviate a map, in the sense that tensor factors that are not underbraced are sent via the identity map, e.g.,
\eq
\xymatrixrowsep{1.4pc}
\xymatrix{
\ud{\tau^B}{W \cB} \cM 
\ar[d]
\\
W \cB \cM 
}
\quad
\textup{ is a shorthand for the map}
\quad
\tau^B \id_\cM  : W \cB  \cM  \to  \cB W \cM.
\endeq

\subsubsection{Necessary Conditions (i)}\label{app:neci}
Notice that all $m \in M^{\otimes d}$,
\eq
\tau^M(s_ks_l\otimes m) 
= \sigma^M_k\sigma^M_l(m) \otimes h_k h_l
+  \rho^M_k\sigma^M_l(m) \otimes h_l 
+ \sigma^M_k\rho^M_l(m) )\otimes h_k
+  \rho^M_k\rho^M_l(m) \otimes 1,
\endeq
for $|k-l| >1$.
Also, for $|i-j| =1$, 
\eq
\begin{split}
\tau^M(s_is_js_i\otimes m) 
&= \sigma^M_i\sigma^M_j\sigma^M_i(m)\otimes h_i h_j h_i 
+  \rho^M_i\sigma^M_j\sigma^M_i(m)\otimes h_j h_i
+ \sigma^M_i\sigma^M_j\rho^M_i(m)\otimes h_ih_j
\\
& 
+  \rho^M_i\sigma^M_j\rho^M_i(m)\otimes h_j
+  (\rho^M_i\rho^M_j\sigma^M_i(m)
+ \sigma^M_i\rho^M_j\rho^M_i(m) )\otimes h_i
\\
&
+ \sigma^M_i\rho^M_j\sigma^M_i(m)(\SK\otimes h_{i}+\RK\otimes 1)
+  \rho^M_i\rho^M_j\rho^M_i(m) \otimes 1.
\end{split}
\endeq
Hence, $\tau^M$ is independent of the choice of $r$ only if \eqref{def:taubr1}--\eqref{def:taubr3} hold.

\subsubsection{Necessary Conditions (ii)}\label{app:necii}
Here, we deduce the conditions on $\sigma_i^M$ and $\rho^M_i$ such that  the diagrams in \eqref{fig:basedmod} commute. 

Consider the case $w = s_i$ for all $i$.
By \eqref{def:tauaction}, 
$s_i \otimes \alpha_M(b\otimes m) = \tau^M(s_i \otimes (b\cdot m)) = \sigma^M_i(b\cdot m) \otimes h_i + \rho^M_i(b\cdot m) \otimes1 $. 
On the other hand,
\eq
\begin{split}
\tau^B(s_i \otimes b )\otimes m &= \sigma_i(b)\otimes h_i \otimes m + \rho_i(b)\otimes 1 \otimes m 
\\
&\overset{\id\otimes \tau_i^M}{\mapsto} \sigma_i(b)\otimes\sigma_i^M(m) \otimes h_i+ (\sigma_i(b)\otimes\rho^M_i(m) + \rho_i(b)\otimes m)\otimes 1
\\
&\overset{\alpha_M\otimes \id}{\mapsto} \sigma_i(b)\sigma_i^M(m)\otimes h_i+ (\sigma_i(b)\rho^M_i(m) + \rho_i(b)m)\otimes 1.
\end{split}
\endeq
Therefore, the commutativity of the first diagram implies that \eqref{def:M2} holds. 
For the second diagram, 
\eq
\begin{split}
m_W(s_i\otimes s_i) \otimes m &= S_i \otimes s_i \otimes m + R_i \otimes 1\otimes m
\\
&\overset{\id\otimes\tau^M}{\mapsto} S_i \otimes (\sigma_i^M(m) \otimes h_i + \rho^M_i(m) \otimes 1) + R_i \otimes m \otimes 1
\\
&\overset{\alpha_M\otimes \id}{\mapsto} \SK \sigma_i^M(m) \otimes h_i + (\rho^M_i(m) + \RK  m) \otimes 1.
\end{split}
\endeq 
On the other hand,
\eq
\begin{split}
&s_i\otimes \tau^M(s_i\otimes m) = s_i \otimes (\sigma^M_i(m)\otimes h_i + \rho^M_i(m) \otimes 1)  
\\
~~&\overset{\tau^M\otimes \id}{\mapsto}
 (\sigma^M_i\sigma^M_i(m)\otimes h_i^2+\rho^M_i\sigma^M_i(m)\otimes h_i + \sigma^M_i\rho^M_i(m)\otimes h_i + \rho^M_i\rho^M_i(m)\otimes 1
\\
~~&\overset{\id\otimes m_\cH}{\mapsto} (\SK\sigma^M_i\sigma^M_i(m)+ \rho^M_i\sigma^M_i(m) + \sigma^M_i\rho^M_i(m))\otimes h_i + \rho^M_i\rho^M_i(m)+\RK \sigma^M_i\sigma^M_i(m).
\end{split}
\endeq
In other words, \eqref{def:M4} must hold.

\subsubsection{Independence of Choice of $r$}\label{app:indep}
Note that any pair of expressions of $w$ is differed from each other by a chain of braid relations. 
It then suffices to show that, for two expressions $r,r'$ of $w$ which differ by a single braid
relation,
\eq\label{eq:indtau}
(\id\otimes\mu)\circ(\tau^M\otimes\id)(s_i\otimes \tau^M(x \otimes m))
= (\id\otimes\mu)\circ(\tau^M\otimes\id)(s_j\otimes \tau^M(x' \otimes m)),
\endeq
where $s_i$ (resp. $s_j$) is the first transposition of $r$ (resp. $r'$).
We may assume that $i \neq j$, and hence we can write $w =  \beta_{i,j} y = \beta_{j,i}y$ for some subexpression $y$ such that$\beta_{i,j} := s_is_j $ if $|i-j|>1$, and $\beta_{i,j} := s_is_js_i$ if $|i-j|=1$.

Let $m_w \in \cM$ be such that $\tau^M(y\otimes m) = \sum_w m_w \otimes h_w$. 
Note that $\tau^M(y\otimes m)$ does not depend on the choice of $r$ since $\ell(y) < \ell(w)$.

For the first case that $|i-j| >1$, 
\eq
\tau^M(s_j y \otimes m) = \sum_w (1\otimes \mu)(\tau^M(s_j \otimes  m_w) \otimes h_w)
= \sum_w \sigma_j^M(m_w) \otimes h_{s_jw} + \rho_j^M(m_w) \otimes h_{w},
\endeq
and hence the left-hand side of \eqref{eq:indtau} is equal to
\eq
\sum_w (
\sigma_i^M\sigma_j^M(m_w) \otimes h_{s_is_jw} 
+\rho_i^M\sigma_j^M(m_w) \otimes h_{s_jw} 
+\sigma_i^M\rho_j^M(m_w) \otimes h_{s_iw} 
+ \rho^M_i\rho_j^M(m_w) \otimes h_{w}) 
,
\endeq
which equals to the right-hand side of \eqref{eq:indtau},
thanks to that these endomorphisms commute.

For the second case that $|i-j| =1$, the left-hand side of \eqref{eq:indtau} is then equal to
\eq
\begin{split}
&\sum_w (\sigma^M_i\sigma^M_j\sigma^M_i(m_w)\otimes h_{s_i s_j s_i w} 
+  \rho^M_i\sigma^M_j\sigma^M_i(m_w)\otimes h_{s_j s_i w}
+ \sigma^M_i\sigma^M_j\rho^M_i(m_w)\otimes h_{s_i s_j  w}
\\
& 
~~+  \rho^M_i\sigma^M_j\rho^M_i(m_w)\otimes h_{s_j w}
+  (\rho^M_i\rho^M_j\sigma^M_i(m_w)
+ \sigma^M_i\rho^M_j\rho^M_i(m_w) )\otimes h_{s_i w}
\\
&
~~+ \sigma^M_i\rho^M_j\sigma^M_i(m_w)(\SK\otimes h_{s_i x}+\RK\otimes h_{w})
+  \rho^M_i\rho^M_j\rho^M_i(m_w) \otimes h_w),
\end{split}
\endeq
which equals to the right-hand side of \eqref{eq:indtau},
thanks to \eqref{def:taubr1}--\eqref{def:taubr3}.


\endproof

\subsubsection{Commutativity of (WB)}\label{app:WBcom}
We begin with specifying factors of the two composite maps in (WB$^{\ell+1}$) using \eqref{def:tauMdiag}. We obtain the two composite maps corresponding to the leftmost and the rightmost paths in the diagram below:
\eq\label{fig:WBell}
\xymatrixrowsep{0.6pc}
\xymatrixcolsep{1pc}
\xymatrix{
&\ud{\iota_r}{W^{(\ell+1)}} \cB \cM \ar[d] 		
\\
W^{(1)} W^{(\ell)} \ud{\alpha_M}{\cB \cM} \ar[d] \ar@{=}[r]
	& W^{(1)} \ud{\tau^B}{W^{(\ell)} \cB} \cM \ar[d] 
		& 
\\
W^{(1)} \ud{\tau^M}{W^{(\ell)} \cM} \ar[d]
	& W^{(1)} \cB \ud{\tau^M}{ W^{(\ell)} \cM} \ar[d] \ar@{=}[r]
		&W^{(1)} \cB \ud{\tau^M}{ W^{(\ell)} \cM} \ar[d] \ar@{=}[r]
			&\ud{\tau^B}{W^{(1)} \cB}  W^{(\ell)} \cM \ar[d]
\\
W^{(1)} \cM \tcH  \ar@{=}[dr]
	&W^{(1)} \ud{\alpha_M} {\cB \cM} \tcH \ar[d] \ar@{=}[r]
		&\ud{\tau^B}{W^{(1)} \cB} \cM \tcH \ar[d]
			& \cB W^{(1)} \ud{\tau^M}{W^{(\ell)} \cM } \ar[dl] 
\\
	&\ud{\tau^M}{W^{(1)} \cM} \tcH  \ar[d]
		& \cB \ud{\tau^M}{W^{(1)} \cM} \tcH \ar[d]
\\
	&\cM  \tcH \tcH \ar@{=}[dr]
		& \ud{\alpha_M}{\cB \cM} \tcH \tcH  \ar[d] \ar@{=}[r]
			&\cB \cM \ud{m_\cH}{\tcH \tcH}  \ar[d]
\\
	&
		& \cM \ud{m_\cH}{\tcH \tcH} \ar[d]
			& \ud{\alpha_M}{\cB \cM} \tcH \ar[dl]
\\
	&
		&\cM \tcH
}
\endeq
First, note that (WB$^1$) commutes.
It also follows from inductive hypothesis that (WB$^\ell$) commutes.
Therefore, the following corresponding subdiagrams commute:
\eq
\xymatrixrowsep{1pc}
\xymatrixcolsep{1pc}
\xymatrix{
W^{(1)} \ud{\alpha_M} {\cB \cM} \tcH \ar[d] \ar@{=}[r]
	&\ud{\tau^B}{W^{(1)} \cB} \cM \tcH \ar[d]
\\
\ud{\tau^M}{W^{(1)} \cM} \tcH  \ar[d]
	& \cB \ud{\tau^M}{W^{(1)} \cM} \tcH \ar[d]
\\
\cM  \tcH \tcH \ar@{=}[dr]
	& \ud{\alpha_M}{\cB \cM} \tcH \tcH  \ar[d] 
\\
	& \cM \tcH \tcH
}
\qquad
\xymatrix{
W^{(1)} W^{(\ell)} \ud{\alpha_M}{\cB \cM} \ar[d] \ar@{=}[r]
	& W^{(1)} \ud{\tau^B}{W^{(\ell)} \cB} \cM \ar[d] 
\\
W^{(1)} \ud{\tau^M}{W^{(\ell)} \cM} \ar[d]
	& W^{(1)} \cB \ud{\tau^M}{ W^{(\ell)} \cM} \ar[d] 
\\
W^{(1)} \cM \tcH  \ar@{=}[dr]
	&W^{(1)} \ud{\alpha_M} {\cB \cM} \tcH \ar[d] 
\\
	&W^{(1)} \cM \tcH  
}
\endeq
Note that in any composite map produced from a path $W^{(1)} \cB W^{(\ell)} \cM \to \cM \tcH$, one can swap two adjacent factor maps as long as they act on unrelated tensor factors, e.g.,
$\tau^B \id_{\cM \tcH} \circ \id_{W^{(1)} \cB} \tau^M_\ell = \id_{\cB W^{(\ell)} \cM}\circ \tau^B \id_{W^{(\ell)} \cM} : W^{(1)} \cB W^{(\ell)} \cM \to \cB W^{(1)} \cM \tcH$. 
Therefore, the entire diagram \eqref{fig:WBell} commutes.

\subsubsection{Commutativity of (WW)}\label{app:WWcom}
Consider the following diagram:
\eq\label{fig:WWell}
\xymatrixrowsep{0.8pc}
\xymatrixcolsep{0.6pc}
\xymatrix{
W^{(k)} \ud{\iota_r}{W^{(\ell+1)}} \cM \ar[d]
\\
W^{(k)} \ud{\substack{\id m_\cH \circ \tau^M\id \circ \id\tau^M\\=\alpha_M\id \circ \id\tau^M \circ m_W\id}}{W^{(1)} W^{(\ell)} \cM} \ar[d] \ar@{=}[r]
	&W^{(k)} \ud{\id\tau^M \circ m_W\id}{W^{(1)} W^{(\ell)} \cM} \ar[d] \ar@{=}[r]
		& \ud{\tau^B\id\circ \id m_W}{W^{(k)} W^{(1)} W^{(\ell)}} \cM \ar[d]  \ar@{=}[r]
			& \ud{\substack{\id\id\tau^M\circ \id m_W\id \circ\\ ~~~\tau^B\id\id\circ \id m_W\id}}{W^{(k)} W^{(1)} W^{(\ell)} \cM } \ar[d] \ar@{=}[r]
				&  \ud{\substack{m_B\id \circ \id m_W  \circ\\~~~ \tau^B \id \circ \id m_W }}{W^{(k)} W^{(1)} W^{(\ell)}} \cM  \ar[d]
\\
\ud{\tau^M}{W^{(k)} \cM} \tcH \ar[dr]
	& \ud{\substack{\tau^M\circ \id\alpha_M\\ =\alpha_M\id\circ\id\tau^M\circ\tau^B\id}}{W^{(k)} \cB \cM} \tcH \ar[d]
		&\cB \ud{\substack{\id m_\cH \circ \tau^M\id \circ \id\tau^M\\=\alpha_M\id \circ \id\tau^M \circ m_W\id}}{W^{(k)} W^{(\ell)} \cM} \ar[d]
			&\ud{\substack{\alpha_M \circ \id\alpha_M \\ = \alpha_M \circ m_B \id}}{\cB \cB \cM} \tcH \ar[ddl] 
				&\cB \ud{\tau^M}{W^{(\ell+k)} \cM} \ar[d]
\\
	&\cM \ud{m_\cH}{\tcH \tcH} \ar[dr]
		& \ud{\alpha_M}{\cB \cM} \tcH \ar[d]
			&
				&\ud{\alpha_M}{\cB \cM} \tcH \ar[dll]
\\
	&
		&\cM \tcH
			&
				&
}
\endeq
Thanks to the the commutativity of (WW$_k^\ell$), we have the identification of maps 
$\id m_\cH \circ \tau^M\id \circ \id\tau^M =\alpha_M\id \circ \id\tau^M \circ m_W\id: W^{(k)} W^{(\ell)} \cM \to \cM \tcH$.
Since (WB$^k$) commutes, 
$\tau^M\circ \id\alpha_M =\alpha_M\id\circ\id\tau^M\circ\tau^B\id : W^{(k)} \cB \cM \to \cM \tcH$.
Additionally, $\alpha_M \circ \id \alpha_M = \alpha_M\circ m_B\id : \cB \cB \cM \to \cM$ since $M$ is a $B$-module.
Therefore, the diagram \eqref{fig:WWell} is well-defined.

Now, there are five paths from $W^{(k)} W^{(\ell+1)} \cM$ to $\cM \tcH$.
To show the the diagram \eqref{fig:bigcomm} commutes, we need to check that the composite map produced from any such path coincides with the composite map produced from the path to its right. 

First, the composite maps $W^{(k)} W^{(1)} W^{(\ell)} \cM \to \cM\tcH$ corresponding to the leftmost and second leftmost paths coincide, since the factor maps are composed in the exact same order.

Then, recall that in any composite map produced from such a path, one can swap two adjacent factor maps as long as they act on unrelated tensor factors.
Hence, the following diagrams commute:
\eq\label{fig:appcomm1}
\xymatrixrowsep{1pc}
\xymatrixcolsep{1pc}
\xymatrix{
W^{(k)} \cB \ud{\tau^M}{W^{(\ell)} \cM} \ar[d] \ar@{=}[r]
	& \ud{\tau^B}{W^{(k)} \cB} W^{(\ell)} \cM \ar[d] 
\\
\ud{\tau^B}{W^{(k)} \cB}  \cM \tcH \ar[dr] 
	& \cB  W^{(k)}\ud{\tau^M}{W^{(\ell)} \cM}  \ar[d] 
\\
	&\cB  \ud{\tau^M}{W^{(k)}\cM} \tcH \ar[d]
\\
\ud{\alpha_M}{\cB \cM} \tcH \tcH \ar[d] \ar@{=}[r]
	& \cB \cM \ud{m_\cH}{\tcH \tcH} \ar[d]
\\
 \cM \ud{m_\cH}{\tcH \tcH} \ar[dr]
	&\ud{\alpha_M}{\cB \cM} \tcH \ar[d]
\\	
&\cM \tcH
}
\qquad
\xymatrix{
\ud{\id m_W \circ \tau^B\id\circ\id m_W}{W^{(k)} W^{(1)} W^{(\ell)}} \cM \ar[d]
\\
\cB \cB \ud{\tau^M}{W^{(k+\ell)} \cM} \ar[d] \ar@{=}[r]
	&\ud{m_B}{\cB \cB} W^{(k+\ell)} \cM  \ar[d]
\\
\ud{m_B}{\cB \cB} \cM \tcH \ar[dr]
	& \cB  \ud{\tau^M}{W^{(k+\ell)} \cM}  \ar[d]
\\
	&\ud{\alpha_M}{\cB \cM} \tcH \ar[d]
\\	
	&\cM \tcH
}
\endeq
As a result, the entire diagram \eqref{fig:WWell} commutes.
Finally, \eqref{eq:doubleinduction} holds since $m_W=m_B\id \circ \id m_W \id \circ \tau^B \id \circ \id m_W \id : W^{(k)}W^{(\ell+1)} \to \cB W$.

\subsection{Verification of Commutativity of \eqref{fig:bigcomm}}\label{app:bigcomm}
Once again, in any composite map produced from a path $\cB W \cB W \cM N \to \cM N$, one can swap two adjacent factor maps as long as they act on unrelated tensor factors.
Thus, all four diagrams below commutes:
\eq
\xymatrixrowsep{1pc}
\xymatrixcolsep{1pc}
\xymatrix{
\cB W \ud{\alpha_{MN}}{\cB W \cM N} \ar[dd] \ar@{=}[r]
	& \cB W \cB \ud{\tau^M}{W \cM} N \ar[d] 
\\
	& \cB W \ud{\alpha_M}{\cB \cM} \ud{\alpha_N}{\tcH N}\ar[d] \ar@{=}[r]
		& \cB \ud{\tau^M\circ\id\alpha_M}{W \cB \cM} \tcH  N \ar[d]
\\	
\ud{\alpha_{MN}}{\cB W \cM N}\ar[ddr] \ar@{=}[r]
	& \cB \ud{\tau^M}{W \cM} N\ar[d]
		&\cB \cM \tcH \ud{\alpha_N}{\tcH N}\ar[dl]
\\
	& \ud{\alpha_M}{\cB \cM} \ud{\alpha_N}{\tcH N}\ar[d]
\\
	& \cM N
}
\qquad
\xymatrix{
\cB W \cB \ud{\tau^M}{W \cM} N \ar[d] \ar@{=}[r]
	&  \cB \ud{\tau^B}{W \cB} \ud{\tau^M}{W \cM} N \ar[ddl]
\\
\cB \ud{\tau^B}{W \cB} \cM \tcH  N \ar[d]
	&
\\
\cB \ud{\alpha_M\circ\id\tau^M}{\cB W \cM} \tcH  N \ar[d] 
\\
\cB  \cM \tcH \ud{\alpha_N}{\tcH  N} \ar[d] \ar@{=}[r]
	&\ud{\alpha_M}{\cB \cM}  \tcH \ud{\alpha_N}{\tcH  N} \ar[dl]
\\
\ud{\alpha_M}{\cB \cM}  \ud{\alpha_N}{\tcH  N} \ar[d]
\\
\cM N
}
\endeq
\eq\label{fig:appcomm2}
\xymatrixrowsep{1pc}
\xymatrixcolsep{1pc}
\xymatrix{
\cB W \cB \ud{\tau^M}{W \cM} N \ar[d] \ar@{=}[r]
	&\cB\ud{\tau^B}{W \cB} W \cM N \ar[d]
\\
\cB \ud{\tau^B}{W \cB} M \tcH N \ar[d] 
	&\ud{m_B}{\cB \cB} W W \cM N \ar[d]
\\
\cB \cB \ud{\tau^M}{W \cM} \tcH N \ar[d]
	&\cB W \ud{\tau^M}{W \cM} N \ar[d]
\\
\ud{m_B}{\cB \cB} \cM \tcH \tcH N \ar[dr]
	&\cB \ud{\tau^M}{W M} \tcH  N \ar[d]
\\
	&\ud{\alpha_M}{\cB \cM}  \udr{\alpha_N \circ m_\cH\id}{\tcH \tcH N} \ar[d]
\\
	& \cM N
}
\qquad
\xymatrix{
\ud{ m_B \circ \id\tau^B}{\cB W \cB} W \cM N \ar[d] \ar@{=}[r]
	& \cB \ud{\tau^B}{W \cB} W \cM N \ar[d] \ar@{=}[r]
		& \ud{ m}{\cB W \cB W} \cM N \ar[ddd] 
\\
\cB \ud{m_W}{WW} \cM N \ar[dr]
	&\ud{m_B}{\cB \cB} \ud{m_W}{W W} \cM N \ar[d]
		&
\\
\cB \cB \ud{\tau^M}{W \cM} N \ar[dr]  \ar@{=}[r]
	&\ud{m_B}{\cB \cB} W \cM N \ar[dr] 
\\
\cB \ud{\alpha_M}{\cB \cM} \tcH N \ar[dr] \ar@{=}[r]
	& \ud{m_B}{\cB \cB} \cM \tcH N \ar[d]
		& \cB \ud{\tau^M}{W \cM} N \ar[dl]
\\
	&\ud{\alpha_M}{\cB \cM} \ud{\alpha_N}{\tcH N} \ar[d]
\\
	&\cM N
}
\endeq
As a result, the diagram in \eqref{fig:bigcomm} commutes.

\subsection{Proof of Theorem \ref{thm:indhom}} \label{app:indhom}

Since $M$ is a \based module, \hyperref[WB]{(WB)} translates to the following commutative diagram:
\eq\label{fig:BWcomm}
\xymatrixrowsep{1pc}
\xymatrixcolsep{1pc}
\xymatrix{
\ud{(\tau^B)\inv}{\cB W} \cM \ar@{=}[r] \ar[d]
	& \cB \ud{\tau^M}{W \cM} \ar[d]
\\
W \ud{\alpha_M}{\cB \cM} \ar[d]
	& \ud{\alpha_M}{\cB \cM} \tcH \ar[d]
\\
\ud{\tau^M}{W \cM} \ar[r]
	& \cM \tcH
}
\endeq
It suffices to check that
$(b H_w) \cdot \tau^M( H_x \otimes m) =  \tau^M( (b H_w) \cdot (H_x \otimes m))$ for all $b \in \cB$, $w, x \in \Sigma_d$, $m\in \cM$,
which will follow as long as the diagram below commutes:
\eq\label{fig:BWWM}
\xymatrixrowsep{1pc}
\xymatrixcolsep{1pc}
\xymatrix{
\cB W \ud{\tau^M}{W \cM} \ar[d] \ar@{=}[r]
	& \cB \ud{m_W}{W W} \cM \ar[d] \ar@{=}[rr]
		&
			&\cB \ud{\id\alpha_M \circ (\tau^B)\inv\id \circ m_W\id}{W W \cM} \ar[ddd]
\\
\cB \ud{\tau^M}{W \cM} \tcH	\ar[d]
	&\cB \cB \ud{\tau^M}{W\cM} \ar[d] \ar@{=}[r]
		& \cB \ud{(\tau^B)\inv}{\cB W} \cM \ar[d]
\\
\cB \cM \ud{m_\cH}{\tcH \tcH} \ar[d] 
	&\cB \ud{\alpha_M}{\cB \cM } \tcH \ar[d]
		& \cB W \ud{\alpha_M}{\cB \cM} \ar[d] 
\\
\ud{\alpha_M}{\cB \cM} \tcH \ar@{=}[r] 
	&\ud{\alpha_M}{\cB \cM} \tcH \ar@{=}[dr] 
		&\cB \ud{\tau^M}{W \cM} \ar[d] \ar@{=}[r]
			& \ud{(\tau^B)\inv}{\cB W} \cM \ar[d]
\\
	&
		&\ud{\alpha_M}{\cB \cM} \tcH \ar[d]
			& W \ud{\alpha_M}{\cB \cM} \ar[d]
\\
	&
		&\cM \tcH
			& \ud{\tau^M}{W \cM} \ar[l]
}
\endeq
The commutativity of \eqref{fig:BWWM} follows from the fact that 
both \hyperref[WW]{(WW)}  and \eqref{fig:BWcomm}  commute.

\end{document}